\documentclass[11pt,a4paper,reqno]{article}
\pdfoutput=1
\usepackage[francais, english]{babel}
\usepackage[latin1]{inputenc}
\usepackage[T1]{fontenc}
\usepackage{latexsym,amsfonts,amssymb,amsmath,amsthm,textcomp, mathtools}
\usepackage{longtable}
\usepackage{graphicx}
\usepackage{tikz}
\usepackage{fancyhdr}
\usepackage[hidelinks]{hyperref}
\usepackage{calc}
\usepackage{chngcntr}

\theoremstyle{plain}
\newtheorem{thm}{Th\'eor\`eme}[section]

\newtheorem{definition}[thm]{D\'efinition}
\newtheorem{proposition}[thm]{Proposition}

\newtheorem{corollaire}[thm]{Corollaire}

\newtheorem{théorème}[thm]{Théorème}
\newtheorem{lemme}[thm]{Lemme}
\newtheorem{exemple}[thm]{Exemple}

\newcommand*{\bigboxplus}{\DOTSB\mathop{\mathpalette\big@boxplus\relax}\slimits@}

\theoremstyle{definition}
\newtheorem{notation}[thm]{Notation}

\newtheorem{remarque}[thm]{Remarque}

\counterwithin{equation}{section}

\def \N {\mathbb{N}}
\def \Q {\mathbb{Q}}
\def \Z {\mathbb{Z}}
\def \R {\mathbb{R}}
\def \C {\mathbb{C}}

\def \F {\mathbb{F}}

\def \G {\mathbb{G}}

\def \Gal{\mathop{\mathrm{Gal}}\nolimits}

\def \Bun{\mathop{\mathrm{Bun}}\nolimits}

\def \dim{\mathop{\mathrm{dim}}\nolimits}
\def \dimtrg{\mathop{\mathrm{dim.trg}}\nolimits}

\def \hom{\mathop{\mathrm{Hom}}\nolimits}

\def \lim{\mathop{\mathrm{lim}}\nolimits}

\def \LT{\mathop{\mathrm{LT}}\nolimits}

\def \Mod{\mathop{\mathrm{Mod}}\nolimits}
\def \Perf{\mathop{\mathrm{Perf}}\nolimits}
\def \Proj{\mathop{\mathrm{Proj}}\nolimits}

\def \Id{\mathop{\mathrm{Id}}\nolimits}
\def \id{\mathop{\mathrm{Id}}\nolimits}
\def \Ind{\mathop{\mathrm{Ind}}\nolimits}

\def \Rep{\mathop{\mathrm{Rep}}\nolimits}
\def \rec{\mathop{\mathrm{rec}}\nolimits}

\def \Sht{\mathop{\mathrm{Sht}}\nolimits}
\def \spec{\mathop{\mathrm{Spec}}\nolimits}
\def \Spa{\mathop{\mathrm{Spa}}\nolimits}
\def \Spd{\mathop{\mathrm{Spd}}\nolimits}
\def \St{\mathop{\mathrm{St}}\nolimits}
\def \Speh{\mathop{\mathrm{Speh}}\nolimits}
\def \Res{\mathop{\mathrm{Res}}\nolimits}
\def \J{\mathop{\mathrm{J}}\nolimits}
\def \Spf{\mathop{\mathrm{Spf}}\nolimits}

\def \JL{\mathop{\mathrm{JL}}\nolimits}

\begin{document}
\setcounter{tocdepth}{1}
\title{Shtukas adiques, modifications et applications}{\tiny }
\author{NGUYEN Kieu Hieu}
\maketitle

\newtheorem*{Résumé}{Résumé}
\begin{Résumé}
	 Dans cet article, via l'étude des modifications de fibrés vectoriels sur la courbe de Fargues-Fontaine, on prouve une formule géométrique reliant les tours de Lubin-Tate avec les espaces de Rapoport-Zink non ramifiés simples basiques de type EL de signature $(1, n-1), (p_1, q_1), \cdots, (p_k, q_k)$ où $p_iq_i = 0$. En particulier, on en déduit le calcul des groupes de cohomologie de ces derniers.   	  
\end{Résumé}
\newtheorem*{Abstract}{Abstract}
\begin{Abstract}	
	In this paper, via the study of the modifications of vector bundles on the Fargues-Fontaine curve, we prove a geometric formula relating the Lubin-Tate towers with the simple basic unramified Rapoport-Zink spaces of EL type of signature  $ (1, n-1), (p_1, q_1), \cdots, (p_k, q_k) $ where $ p_iq_i = 0 $. In particular, we deduce the computation of the cohomology groups of the latter.
\end{Abstract}



\thispagestyle{fancy}
\fancyhf{}
\fancyfoot[L]{\textbf{Classification mathématique par sujet (2010)} 11F70, 11F80, 11F85, 11G18, 20C08  }

\tableofcontents
\section{Introduction}
Le programme de Langlands prédit une bijection de nature arithmétique entre d'un côté les représentations galoisiennes et de l'autre les représentations automorphes. Pour construire des réalisations, on cherche des objets géométriques dont la cohomologie $\ell$-adique s'exprimera en termes de ces correspondances. Les premiers exemples sont donnés par les espaces de Rapoport-Zink (\cite{RZ96}). Par ailleurs, plus récemment, Scholze a construit des variétés de Shimura locales (\cite{RV14}, \cite{SW17}) qui se décrivent comme des espaces de modules de Shtukas dont la partie supercuspidale de la cohomologie est décrite par la conjecture de Kottwitz. 

Commençons par un triplet $(G, \mu, b)$ où $G$ est un groupe réductif sur $\Q_p$, $ b \in G(\breve{\Q}_p) $ et $\mu \in X^{+}_{*}(T)$ avec $b \in B(G, \mu)$ (c.f. section \ref{itm : Grassmanienne affine}) auquel on associe un $G$-fibré $\mathcal{E}_b$ sur la courbe de Fargues-Fontaine. De manière informelle, l'espace de module de Shtukas $\Sht(G, \mu, b)$ est un faisceau sur $\Perf_{\overline{\F}_p}$ classifiant les modifications de type $\mu$ entre $\mathcal{E}_b$ et $\mathcal{E}_1$. L'espace $\Sht(G, \mu, b)$ admet une donnée de descente ainsi qu'une action du groupe $G(\Q_p)$ (resp. $\J_b(\Q_p)$) définie par les isomorphismes de $\mathcal{E}_1$ (resp. $\mathcal{E}_b$). Le premier résultat principal de cet article est le théorème \ref{itm : géométrique} suivant.
\begin{théorème} \phantomsection \label{itm : thm}
	Notons $Z^0_G$ la composante connexe neutre du centre $Z_G$, $\lambda$ un cocaractère central de $G$ et $[b_{\lambda}]$ l'unique élément de $B(Z^0_G, \lambda)$. Il y a un isomorphisme $G(\Q_p) \times \J_b(\Q_p)$-équivariant de faisceaux pro-étale qui commute avec les données de descentes.
	\[
	\Sht(G, \mu, b) \times_{\underline{Z_G^0(\Q_p)}} \Sht(Z_G^0, \lambda) \longrightarrow \Sht(G, b\overline{b}_{\lambda} ,\mu \cdot \lambda).
	\]
\end{théorème}
Le résultat ci-dessus nous permet d'interpréter l'espace $\Sht(G, b\overline{b}_{\lambda} ,\mu \cdot ( \lambda))$ comme une modification centrale de l'espace $\Sht(G, \mu, b)$. La démonstration se base sur l'étude des $G$-fibrés sur la courbe de Fargues-Fontaine. La première étape consiste à construire une modification de type $\mu \cdot \lambda$ à partir de celle de type $\mu$. Ensuite, on propage cette construction au niveau des espaces de modules de Shtukas de manière compatible avec les actions de $G(\Q_p) \times \J_b(\Q_p)$ et notamment avec les données de descentes.

On applique ce résultat géométrique au calcul des groupes de cohomologie d'espaces de Rapoport-Zink que l'on peut ainsi relier avec la tour de Lubin-Tate. Pour chaque sous-groupe compact ouvert $K_p \subset G(\Q_p)$, il y a un espace de module de Shtukas de niveau fini $ \Sht(G, \mu, b) / K_p$ et une identification $ \displaystyle \Sht(G, \mu, b) = \mathop{\mathrm{lim}}_{\overleftarrow{K_p}} \Sht(G, \mu, b)/ K_p$. Lorsque $\mu$ est minuscule et le triplet $(G, \mu, b)$ est de type EL ou PEL, on retrouve les espaces de Rapoport-Zink, d'après le théorème 24.2.5 de \cite{SW17}. Afin de se ramener à la tour de Lubin-Tate qui correspond à la signature $\mu_{\mathcal{LT}} = (1, n-1), (0, n), \cdots, (0, n)$, on considère les espaces de Rapoport-Zink de type EL non ramifiés simples basiques de signature $(1, n-1), (p_1, q_1), \cdots, (p_k, q_k)$ où $p_iq_i = 0$. Soit $K_p \subset G(\Q_p) $ un sous groupe compact. On notera $ \Sht(G, \mu, b)/K_p $ par $\breve{\mathcal{M}}_{K_p}^{\mu}$.  D'après \cite{Far04}, on peut définir les groupes de cohomologie à support compact par
\[
H_c^{\bullet}(\breve{\mathcal{M}}_{K_p}^{\mu}, \Z_{\ell}) = \mathop{\mathrm{lim}}_{\overrightarrow{V}} \mathop{\mathrm{lim}}_{\overleftarrow{n}} H_c^{\bullet}(V \otimes_{\breve{\Q}_p} \C_p, \Z / \ell^n \Z )
\]
où $V$ parcourt les ouverts relativement compacts de $\breve{\mathcal{M}}^{\mu}_{K_p}$ et où $\ell \neq p$ est un nombre premier. On note également $H_c^{\bullet}(\breve{\mathcal{M}}^{\mu}_{K_p}, \Q_{\ell}) := H_c^{\bullet}(\breve{\mathcal{M}}_{K_p}^{\mu}, \Z_{\ell}) \otimes \overline{\Q}_{\ell}$ lesquels sont des $\overline{\Q}_{\ell}$-représentations de $G(\Q_p) \times D^{\times} \times W_{E_p}$ où $\J_b(\Q_p) = D^{\times}$ est le groupe des inversibles de l'algèbre de division d'invariant $\dfrac{1}{n}$ et où $W_{E_p}$ est le groupe de Weil du corps de définition $E_p$ de $\mu$.

  Les groupes de cohomologie de la tour de Lubin-Tate ont été entièrement calculés dans \cite{Boyer99}, \cite{HT01} et \cite{Boyer09}; par dualité \cite{Fal} \cite{FGL} \cite{SW17}, on obtient aussi le cas de Drinfeld. Pour un espace de type EL non ramifié général, la partie supercuspidale est traitée dans \cite{Far04}, \cite{Shin}. Sous l'hypothèse précédente, on obtient alors la description complète de groupes de cohomologie suivante avec les notations de la section \ref{itm : application}.

\begin{théorème}
		Pour tout diviseur $g$ de $n = gs$ et toute représentation irréductible cuspidale $\pi$ de $GL_g(F)$, on a des isomorphismes $G(\Q_p) \times W_{F}$-équivariants
		\[
		\mathop{\mathrm{lim}}_{\overrightarrow{K}} H_c^{n-1-i}(\breve{\mathcal{M}}_K^{\mu})[\pi[s]_D] = \left\lbrace \begin{array}{cc}
		\displaystyle \LT_{\pi}(s,i) \otimes \mathcal{L}(\pi) | \cdot |^{ - \frac{s(g+1) - 2(i+1)}{2}} \cdot \prod_{\tau \in J}  \omega_i \circ (\prescript{\tau}{}{\rec}^{-1}_{F}) & 0 \leq i < s \\
		0 & i < 0
		\end{array} \right. 
		\]
		où $\pi[s]_D$ désignera la représentation $\JL^{-1}(\St_s(\pi))^{\vee}$ de $D^{\times}$, $\omega_i$ est le caractère central de $\LT_{\pi}(s,i)$ et $\rec^{-1}_{F}$ est le morphisme de réciprocité d'Artin et où $\prescript{\tau}{}{\rec}^{-1}_{F} := \tau \cdot \rec^{-1}_{F} \cdot \tau^{-1}$.
\end{théorème}

La preuve repose sur le théorème \ref{itm : thm} qui donne une relation entre la tour $(\breve{\mathcal{M}}^{\mu}_{K_p})_{K_p}$ et la tour de Lubin-Tate ainsi que les résultats dans \cite{Boyer09} décrivant la cohomologie de la dernière.
\newtheorem*{remerciements*}{Remerciements}
\begin{remarque}
	En utilisant les résultats de \cite{Boyer14}, on peut prouver que les groupes de cohomologie $ \displaystyle \mathop{\mathrm{lim}}_{\overrightarrow{K}} H_c^{\bullet}(\breve{\mathcal{M}}^{\mu}_{K_p}, \Z_{\ell}) $ sont $\Z_{\ell}$-libres.
\end{remarque}
\begin{remerciements*}
	Je remercie profondément Pascal Boyer et Laurent Fargues tant pour leur aide mathématique déterminante que pour leurs constants encouragements.
\end{remerciements*}	
\section{La courbe de Fargues-Fontaine d'après \cite{FF}}
Soit $E / \Q_p$ une extension finie de corps résiduel $\F_q$ d'uniformisante $\pi$. Soit $F / \F_q $ un corps parfait muni d'une valuation non-triviale $ v : F \longrightarrow \R \cup \{ \infty \} $.

On pose $\mathcal{E} = W_{\mathcal{O}_E}(F) \Big[\dfrac{1}{\pi} \Big]$ où $W_{\mathcal{O}_E}(F)$ désigne l'anneau des vecteurs de Witt ramifiés à coefficients dans $F$. Plus précisément, on a
\[
\mathcal{E} = \Big\{ \sum_{n \gg -\infty} [x_n] \pi^n \ | \ x_n \in F \Big\}.
\]  

Le corps $\mathcal{E}$ est une extension non ramifiée complète de $E$ dont le corps résiduel est $F$. Il y a un relèvement de Teichmüler $ F \longrightarrow \mathcal{O}_{\mathcal{E}}, \ x \longmapsto [x]$.
On introduit alors le sous-anneau de $\mathcal{E}$
\[
B^{bd} = \Big\{ \sum_{n \gg -\infty} [x_n] \pi^n \in \mathcal{E} \ | \ \exists C, \ \forall n \ |x_n| < C \Big\}.
\]

On définit ensuite des normes de Gauss sur $B^{bd}$. Pour $r \in ] 0, \infty [$ et $ \rho = q^{-r} \in ] 0, 1 [$ on pose pour $ x = \sum_{n} [x_n] \pi^n \in B^{bd} $
\[
v_r(x) = \mathop{\mathrm{inf}}_{n \in \Z} v(x_n) + nr \ \ \ \mathop{\mathrm{et}}\nolimits \ \ \ | x |_{\rho} = \mathop{\mathrm{sup}}_{n \in \Z} |x_n|\rho^n.
\] 
Il s'agit de normes multiplicatives, autrement dit $v_r$ est une valuation sur $B^{bd}$.
\begin{definition}
	Pour $I \subset ]0, 1[$ un intervalle, on note $B_I$ le complété de $B^{bd}$ par rapport aux normes $(| . |_{\rho})_{\rho \in I}$. Notons $B = B_{]0, 1[}$. On a alors $ B = \varprojlim_I B_I $ où $I$ parcourt les intervalles compacts de $ ] 0, 1 [ $.
\end{definition}
L'anneau $\mathcal{E}$ possède un morphisme de Frobenius $\varphi$ tel que
\[
\varphi \Big( \sum_n [x_n] \pi^n \Big) = \sum_n [x_n^q]\pi^n. 
\]

On note également $ \varphi : [0, 1] \xrightarrow{\sim} [0, 1] $ défini par $\varphi(\rho) = \rho^q$. Le Frobenius $\varphi$ de $B^{bd}$ s'étend en un isomorphisme
\[
\varphi : B_I \xrightarrow{\sim} B_{\varphi(I)}.
\]

Ce morphisme de Frobenius induit alors un automorphisme $\varphi$ de $B$ avec $E = B^{\varphi = \Id}$.
\begin{definition} 
La courbe schématique de Fargues-Fontaine est  $X = \mathop{\mathrm{Proj(P)}}\nolimits $ où
	\[
	P = \bigoplus_{d \geq 0} B^{\varphi = \pi^d},
	\]
	vue comme $E$-algèbre graduée.
\end{definition} 
On note $\mathcal{O}_X(1) = \widetilde{P[1]}$ le fibré en droites tautologique sur $X$ et pour $d \in \Z$, $\mathcal{O}_X(d) = \Big( \mathcal{O}_X(1) \Big)^{\otimes d}$. On a alors
\[
P = \bigoplus_{d \geq 0} H^0(X, \mathcal{O}_X(d)).
\]

 On rappelle ensuite la version adique de la courbe de Fargues-Fontaine. Si $I = [\rho_1, \rho_2] \subset ] 0, 1 [$ avec $\rho_1, \rho_2 \in | F^{\times} | $ alors $B_{I}$ est une $E$-algèbre de Banach qui est un anneau principal. On pose alors $Y_{I} = \mathop{\mathrm{Spa}}\nolimits (B_I, B_I^{o})$ comme espace topologique muni d'un préfaisceau d'anneaux.
\begin{théorème} (\cite{Far16} 2.1)
	L'espace $ Y_{I} $ est adique i.e le préfaisceau $ \mathcal{O}_{Y_I} $ est un faisceau.
\end{théorème}
Soient $ I \subset I' \subset ] 0, 1 [ $ alors $Y_I$ est un ouvert rationnel de $Y_{I'}$. Si $ I = [|a|,|b|] $ avec $a, b \in F^{\times}$ alors 
\[
Y_I = Y_{I'} \Big( \dfrac{[a]}{\pi}, \dfrac{\pi}{[b]} \Big).
\]

On pose $Y = \varinjlim_I Y_I$ où $I$ parcourt les intervalles compacts de $ ] 0, 1 [ $ d'extrémités dans $ | F^{\times} | $.

C'est un espace adique tel que \ $ \Gamma(Y, \mathcal{O}_{Y}) = B $. On peut donc voir les éléments de $B_I$ comme les fonctions holomorphes de la variable $\pi$ sur la couronne de rayons définis par $I$. Le Frobenius $\varphi$ de $B^{bd}$ induit un isomorphisme $ \varphi : B_{I} \xrightarrow{\sim} B_{\varphi(I)} $ et on a donc un isomorphisme
\[
\varphi : Y_{\varphi(I)} \xrightarrow{\sim} Y_I.
\] 

En prenant la limite sur les intervalles $I$ on obtient un automorphisme
\[
\varphi : Y \xrightarrow{\sim} Y.
\]

Ce morphisme $\varphi$ agit de manière proprement discontinue sur $Y$ et on définit alors la courbe de Fargues-Fontaine adique de la manière suivante.
\begin{definition}
	La courbe de Fargues-Fontaine adique est $X^{\text{ad}} = Y / \varphi^{\Z}$.
\end{definition}

 On rappelle enfin la version relative de la courbe de Fargues-Fontaine sur un espace affinoïde perfectoïde $ S $. Intuitivement, on peut y penser comme une famille de courbes $ (X_{k(s)})_{s \in S} $.

Soit $S = \Spa(R, R^{+})$ un espace affinoïde perfectoïde sur $\Spa(\F_q)$. Posons
\[
\textbf{A}_S = W_{\mathcal{O}_E}(R^{+}) = \Big\{ \sum_{n \geq 0} [x_n] \pi^n \ | \ x_n \in R^{+} \Big\},
\]
où $W_{\mathcal{O}_E}$ désigne les vecteurs de Witt ramifiés. Notons
\[
B_S^{bd} = \Big( \textbf{A}_S \Big[ \dfrac{1}{\pi} \Big] \Big)^{bd} = \Big\{ \sum_{n \gg - \infty} [x_n] \pi^n \ | \ x_n \in R, \ \sup_n |x_n| < + \infty \Big\}.
\]

Pour $ \rho \in ] 0, 1 [ $, il y a une norme de Gauss sur $B_S^{bd}$ définie par $\| \sum_{n \gg - \infty} [x_n] \pi^n \|_{\rho} = \sup_{n} | x_n | \rho^n$.

De même, pour $ I \subset ] 0, 1 [ $ un intervalle compact, on note $B_{I, S}$ le complété de $B_S^{bd}$ par rapport aux normes $(\rVert \cdot \rVert_{\rho})_{\rho \in I}$. On note $ B_S = \varprojlim_I B_{I, S} $ le complété de $B_S^{bd}$ par rapport aux normes $(\rVert \cdot \rVert_{\rho})_{\rho \in ]0,1[}$.

Comme auparavant, $ Y_{I, S} = \Spa \Big(B_{I, S}, (B_{I, S})^o \Big) $ ainsi que $Y_S = \varinjlim_I Y_{I, S}$ sont des espaces adiques. Finalement, on définit la courbe de Fargues-Fontaine relative comme suit
\[
X_S^{ad} := Y_S / \varphi^{\Z}
\]
où $\varphi$ est le morphisme de Frobenius des vecteurs de Witt usuel
\[
\varphi \Big( \sum_{n \gg - \infty} [x_n]\pi_n \Big) = \sum_{n \gg - \infty} [x^q_n]\pi_n.
\]

Comme auparavant, il y a des des fibrés vectoriels $ \mathcal{O}_{X_S^{ad}}(d) $ pour tout $d \in \Z$ et on définit le schéma
\[
X_S = \mathop{\mathrm{Proj(P)}}\nolimits
\]
où $ P = \bigoplus_{d \geq 0} H^0 (X_S^{ad}, \mathcal{O}_{X_S^{ad}}(d)) $ et $ \mathcal{O}_{X_S^{ad}}(1) = \widetilde{ (B_{S})^{\varphi = \pi} } $. \\

\textbf{Débasculements et diviseurs de Cartier sur la courbe} \\

Soit $R$ une $\F_q$-algèbre perfectoïde. Notons $R^{o} \subset R$ le sous anneau des éléments de puissances bornées ainsi que $R^{oo} \subset R^{o}$ l'ensemble des éléments topologiquement nilpotents. Un élément $f = \sum_{n \geq 0} [a_n] \pi^n \in \textbf{A}_R $ est dit primitif de degré $1$ si $ a_0 \in R^{oo} \cap R^{\times} $ et $a_1 \in R^{o \times}$.
\begin{proposition} (prop. 1.18 de \cite{Far16}, \cite{Fon})
	\begin{enumerate}
		\item[$\bullet$] Soit $R^{\sharp}$ un débasculement de $R$ sur $E$. Alors, le noyau du morphisme de Fontaine
		\[
		\theta : W_{\mathcal{O}_E}(R^o) \twoheadrightarrow R^{\sharp, o}
		\]
		est engendré par un élément primitif de degré $1$.
		\item[$\bullet$] Inversement, si $f \in \textbf{A}_R$ est un élément primitif de degré $1$ alors $W_{\mathcal{O}_E}(R^o) \Big[ \dfrac{1}{\pi} \Big] / f$ est une $E$-algèbre perfectoïde qui est un débasculement de $R$.
	\end{enumerate}
\end{proposition}
On obtient ainsi une bijection entre l'ensemble des débasculements de $R$ sur $E$ et l'ensemble des immersions fermées $T \hookrightarrow Y_R$ définies localement par un élément primitif de degré $1$.

Autrement dit, si $S \in \Perf_{\F_q}$ alors à chaque débasculement $S^{\sharp}$ correspond un diviseur de Cartier $ D : S^{\sharp} \hookrightarrow Y_S $ et de plus ce diviseur de Cartier induit un diviseur de Cartier $\varphi$-invariant
\[
\bigoplus_{k \in \Z} \varphi^{k*} D
\]
et on obtient donc un diviseur de Cartier sur $X_S^{ad}$.
Supposons que $\overline{\F}_q \subset F$. Pour tout $h \in \N^*$, il existe une unique extension non ramifiée $E_h$ de degré $h$ de $E$. Notons $X_{E_h}$ la courbe adique associé à $E_h$, alors
\[
\pi_h : X_{E_h} = X_E \otimes_E E_h \longrightarrow X_E
\]
est un revêtement étale fini de degré $h$. 

Pour tout $\lambda = \dfrac{d}{h} \in \Q$ où $(d, h) = 1$ et $h > 0$, on définit un fibré vectoriel $\mathcal{O}_{X_E}(\lambda)$ par la formule
\[
\mathcal{O}_{X_E} (\lambda) = (\pi_h)_* \mathcal{O}_{X_{E_h}} (d).
\]

On peut définir les fonctions rang et degré sur les classes d'isomorphisme des fibrés vectoriels et puis appliquer le formalisme des filtrations de Harder-Narasimhan sur la catégorie de fibrés vectoriels. Pour $ \lambda = \dfrac{d}{h} $ avec $(d,h) = 1$ alors $\mathcal{O}_X(\lambda)$ est un fibré de degré $d$ et de rang $h$ et donc de pente de Harder-Narasimhan $\mu(\mathcal{O}_X(\lambda)) = \lambda$. 
\begin{definition}
	Un fibré vectoriel non nul $X$ est semi stable si pour tout sous fibré vectoriel strict non nul $X^{'}$ de $X$, on a $\mu(X^{'}) \leq \mu(X) $.
\end{definition}
\begin{théorème} ( \cite{FF} théo 8.5.1 ) \phantomsection \label{itm : clas}
	Supposons $F$ algébriquement clos.
	\begin{enumerate}
		\item[$\bullet$] Les fibrés semi-stables de pente $\lambda$ sur $X$, à isomorphisme près, sont les $\mathcal{O}_X(\lambda)^m$.
		\item[$\bullet$] La filtration de Harder-Narasimhan d'un fibré vectoriel sur $X$ est scindée.
		\item[$\bullet$] Tout fibré vectoriel sur $X$ est isomorphe à un fibré de la forme $\bigoplus_{i = 1}^n \mathcal{O}_X (\lambda_i)$ où $ \lambda_1 \geq \cdots \geq \lambda_n $ est une suite décroissante dans $\Q$.
	\end{enumerate}
\end{théorème}

Notons $\varphi$-$\Mod_B$ la catégorie des $\varphi$-modules libres sur $B$. On a une équivalence de catégories
\begin{align*}
\varphi-\Mod_B & \xrightarrow{\sim} Fib_X \\
(M, \varphi) & \longmapsto \widetilde{\Big( \bigoplus_{d \geq 0} M^{\varphi = \pi^d} \Big)}.
\end{align*}
D'autre part le foncteur sections globales implique une équivalence entre la catégorie des fibrés $\varphi$-équivariants sur $Y$ et celle des $\varphi$-modules sur $B$.
\begin{théorème}(\cite{Far14} théo. 3.5) Supposons $F$ algébriquement clos. Il y a une équivalence de catégories entre faisceaux cohérents sur $X$ et sur $X^{ad}$.  
\end{théorème}
Notons $\varphi$-$\Mod_{\breve{E}}$ la catégorie des isocristaux sur $\breve{E}$ où $\breve{E} := \widehat{E^{\text{nr}}}$. Il y a un foncteur naturel de $\varphi$- $\Mod_{\breve{E}}$ dans $Fib_X$. Soit $(D, \varphi)$ un isocristal, on pose
\[
\mathcal{E} (D, \varphi) = Y \times_{\varphi^{\Z}} D \longrightarrow Y / \varphi^{\Z} = X^{ad}
\]
qui est un fibré vectoriel sur $X^{ad}$. Via $GAGA$ cela correspond au fibré associé au $P$-module gradué $ \bigoplus_{d \geq 0} \big( D \otimes \mathcal{O}(Y) \big)^{\varphi \otimes \varphi = \pi^d} $. Le théorème \ref{itm : clas} implique que le foncteur $ \mathcal{E}( - ) : \varphi-\Mod_{\breve{E}} \longrightarrow Fib_X $ est essentiellement surjectif.
\begin{definition}
Soit $G$ un groupe réductif sur $E$. Notons $\Bun_X$ la catégorie des fibrés vectoriels sur $X$. Un $G$-fibré sur $X$ (ou $X^{ad}$) est un foncteur tensoriel exact
	\[
	\Rep_{E} G \longrightarrow \Bun_X.
	\]
\end{definition}
On considère l'ensemble de Kottwitz $B(G) = G(\breve{E}) / \sigma$-conj des classes d'isomorphismes de $G$-isocristaux. Si $b \in G(\breve{E})$ on peut lui associer un $G$-fibré sur $X$, que on note $\mathcal{E}_b$, par composition
\begin{align*}
\Rep(G) &\longrightarrow  \varphi-\Mod_{\breve{E}}  \xrightarrow{\mathcal{E}(-)} Bun_X \\
(V, \rho) &\longrightarrow  (V_{\breve{E}}, \rho(b)\sigma).
\end{align*}
\begin{théorème} (\cite{Far16} théorème 2.13)
	On suppose $F$ algébriquement clos. Il y a une bijection entre $B(G)$ et l'ensemble des $G$-fibrés.
	\begin{align*}
	B(G) & \xrightarrow{\sim} H^1_{et} (X, G) \\
	[b] & \longmapsto [\mathcal{E}_b].
	\end{align*}
	Cette bijection généralise le théorème \ref{itm : clas}.
\end{théorème}

Si $S \longrightarrow \Spa(\overline{\F}_q)$ un espace affinoïde perfectoïde et $b \in G(\breve{E})$, on définit comme auparavant un $G$-fibré $\mathcal{E}_b$ sur $X_S$,
\begin{align*}
\Rep(G) &\longrightarrow  \varphi-\Mod_{\breve{E}}  \xrightarrow{\mathcal{E}_S(-)} Bun_{X_S} \\
(V, \rho) &\longrightarrow  (V_{\breve{E}_p}, \rho(b)\sigma).
\end{align*}

Par abus de langage, on le note encore $\mathcal{E}_b$. De plus $\mathcal{E}_1$ est le $G$-fibré trivial.
\section{La $B_{dR}$-Grassmanienne affine} \label{itm : Grassmanienne affine}
 
Soient $S = \Spa (R, R^{+})$ affinoïde perfectoïde de caractéristique $p$ et $(S^{\sharp}, \iota)$ un débasculement de $S$ sur $\Q_p$ et donc un diviseur de Cartier $D_{S^{\sharp}} \hookrightarrow X_S$. On a une application surjective
\[
\theta : W(R^0) \longrightarrow (R^{\sharp})^{0}
\]
dont le noyau est engendré par $ \xi \in W(R^{0})$ qui n'est pas un diviseur de zéro. Alors $B^{+}_{dR}(R^{\sharp})$ est défini comme le complété $ \xi $-adique de $W(R^{0})\Big[\dfrac{1}{p}\Big]$ et $B_{dR}(R^{\sharp}) = B^{+}_{dR}(R^{\sharp})[\xi^{-1}]$. Le complété de $X_S$ au-dessus de $D_{S^{\sharp}}$ est alors $\Spf(B^{+}_{dR}(R^{\sharp}))$. Notons $B_e(R^{\sharp}) = H^{0}(X_S \setminus D_{S^{\sharp}}, \mathcal{O}_{X_S})$.

Etant donné un $G$-fibré $\mathcal{E}$, posons $\mathcal{E}_e$ sa restriction sur $X_S \setminus D_{S^{\sharp}} = \spec B_e(R^{\sharp})$ ainsi que $\mathcal{E}^{+}_{dR}$ son complété au-dessus de $D_{S^{\sharp}}$. La proposition suivante implique que $\mathcal{E}$ est déterminé par $\mathcal{E}_e$ et $\mathcal{E}^{+}_{dR}$.
\begin{proposition} \cite{BL95} (Recollement de Beauville-Laszlo)
	La catégorie des $G$-fibrés sur $X_S$ est équivalent à la catégorie des triplets $\Big( \mathcal{E}_e, \mathcal{E}^{+}_{B_{dR}}, \iota \Big)$ où $\mathcal{E}_e$ est un $G$-fibré sur $B_e(R^{\sharp})$, $\mathcal{E}^{+}_{B_{dR}}$ est un $G$-fibré sur $B^{+}_{dR}(R^{\sharp})$ et $\iota : \mathcal{E}_e \otimes_{B_e(R^{\sharp})} B_{dR}(R^{\sharp}) \xrightarrow{\sim} \mathcal{E}^{+}_{B_{dR}} \otimes_{B^{+}_{dR}(R^{\sharp})} B_{dR}(R^{\sharp})$ est un isomorphisme.
\end{proposition}
\begin{definition}
	Soit $\mathcal{E}$ un $G$-fibré sur $X_S$. Une modification de $\mathcal{E}$ au-dessus de $D_{S^{\sharp}}$ est un $G$-fibré $\mathcal{E}'$ avec un isomorphisme
	\[
	\alpha : \mathcal{E} |_{X_S \setminus D_{S^{\sharp}}} \xrightarrow{\sim} \mathcal{E}' |_{X_S \setminus D_{S^{\sharp}}}.
	\]
\end{definition}

En vertu du recollement de Beauville-Laszlo, une modification entre $\mathcal{E} = \Big( \mathcal{E}_e, \mathcal{E}^{+}_{B_{dR}}, \iota \Big)$ et $\mathcal{E}' = \Big( \mathcal{E}'_e, (\mathcal{E}')^{+}_{B_{dR}}, \iota' \Big)$ est donnée par un isomorphisme $ \alpha : \mathcal{E}_e \xrightarrow{\sim} \mathcal{E}'_e $.

Supposons maintenant $R^{\sharp} = C$ est un corps complet algébriquement clos. On a
\begin{align*}
X_{C^b}  = \mathop{\mathrm{Proj \Big( \bigoplus_{d \geq 0} (B_{C^b})^{\varphi = \pi^d} }}\nolimits \Big). 
\end{align*}
et le diviseur $ D_{C} \hookrightarrow X_{C^b} $ correspond à une injection $ \mathcal{O}_{X_{C^b}} \hookrightarrow \mathcal{O}_{X_{C^b}}(1) $ des fibrés en droites sur $X_{C^b}$ et donc un élément $t \in H^0(X_{C^b}, \mathcal{O}_{X_{C^b}}(1)) = (B_{X_{C^b}})^{\varphi = \pi} $.

Maintenant comme $B_e(C) = H^{0}(X_{C^b} \setminus D_{C}, \mathcal{O}_{X_S}) $ on voit que $ B_e(C) = \spec \Big( (B_{C^b}[t^{-1}])^{\varphi = \Id} \Big) $.

Soit $(D, \varphi)$ un isocristal, on y a associé un fibré vectoriel $ \mathcal{E}(D, \varphi) $. En terme du recollement de Beauville-Laszlo, $ \mathcal{E}(D, \varphi) $ est donné par
\[
\Big( (B_{C^b}[t^{-1}] \otimes_{\breve{E}} D )^{\varphi \otimes \varphi = \Id}  , (B^{+}_{dR}(C))^n, \iota \Big)
\] 
où $\iota$ est le morphisme trivial 
\[
(B_{dR}(C))^n \xrightarrow{\sim} (B^{+}_{dR}(C))^n \otimes B_{dR}(C).
\]

On suppose désormais que $G$ est un groupe réductif connexe défini sur $\Q_p$.

\begin{definition} \cite{SW17}
	La $B_{dR}$-Grassmanienne affine $Gr^{B_{dR}}_{G}$ est la faisceautisation étale du foncteur qui à un espace affinoïde perfectoïde $ S = \Spa(R, R^{+})$ avec un morphisme $S \longrightarrow \Spd \breve{\Q}_p$ correspondant à un débasculement $ S^{\sharp} = \Spa(R^{\sharp}, (R^{\sharp})^{+})$, associe l'ensemble des classes d'isomorphismes de $G$-torseurs sur $\spec B^{+}_{dR}(R^{\sharp})$ trivialisés sur $\spec B_{dR}(R^{\sharp})$.	
\end{definition}

Comme tout $G$-torseur sur $\Spa(R^{\sharp}, (R^{\sharp})^{+})$ devient trivial sur un recouvrement étale de $\Spa(R^{\sharp}, (R^{\sharp})^{+})$, on en déduit que $Gr^{B_{dR}}_G$ est le faisceau étale associé au préfaisceau
\[
(R^{\sharp}, (R^{\sharp})^{+}) \longmapsto G(B_{dR}(R^{\sharp})) / G(B^{+}_{dR}(R^{\sharp})).
\]

Notons $G^*$ le groupe réductif forme quasi-déployée de $G$ (i.e $G_{\overline{\Q}_p} \simeq G^*_{\overline{\Q}_p} $) ainsi que $T \subset B \subset G^*$ un tore maximal contenu dans un sous groupe de Borel de $G^*$. 

Posons $X_*(G) := \hom (\G_{m, \overline{\Q}_p}, G_{\overline{\Q}_p})$. Le groupe de Galois $\Gamma = \Gal(\overline{\Q}_p / \Q_p) $ ainsi que $G(\overline{\Q}_p)$ agissent sur $X_*(G)$ et on a 
\[
\Gamma \backslash \Big[ X_*(G^*) / G^*(\overline{\Q}_p) \Big] = \Gamma \backslash\Big[ X_*(G) / G(\overline{\Q}_p) \Big] = \Gamma\backslash X_*(T)^+.
\]

Pour $S = \Spa (C, C^+)$ avec $C / \Q_p $ un corps perfectoïde algébriquement clos alors $B^{+}_{dR}(C)$ est un anneau de valuation discrète et on voit que
\[
Gr^{B_{dR}}_{G}(C, C^+) = G(B_{dR}(C)) / G(B^{+}_{dR}(C)).
\]

On a la décomposition de Cartan
\[
G(B_{dR}(C)) = \bigsqcup_{\mu \in X_{*}(T)^+} G(B^{+}_{dR}(C)) \mu(\xi)^{-1} G(B^{+}_{dR}(C)). 
\]
\begin{notation}
Notons $\leq$ l'ordre de Bruhat sur $X_{*}(T)^+$, i.e $\mu' \leq \mu$ si et seulement si $ \mu - \mu' $ est une somme des co-racines positives avec des coefficients positifs rationnels.	
\end{notation}
\begin{definition} \cite{SW17} - 19.2.2. (variétés de Schubert).
	
	Pour un cocaractère $\mu \in X_{*}(T)^+$, notons $E$ son corps de définition et $ \breve{E} := E \cdot \breve{\Q} $. Considérons les sous-foncteurs
	\[
	Gr^{B_{dR}}_{G, \mu} \subset Gr^{B_{dR}}_{G, \leq \mu} \subset Gr^{B_{dR}}_G
	\]
	qui sont définis par les conditions qu'un morphisme $S \longrightarrow Gr^{B_{dR}}_G$ avec un morphisme $S \longrightarrow \Spd\breve{E}$ où $S \in \Perf_{\overline{\F}_p}$ se factorise par $Gr^{B_{dR}}_{G,  \mu}$ resp. $Gr^{B_{dR}}_{G, \leq \mu}$ si et seulement si pour tout point géométrique $ x = \Spa(C(x), C(x)^+) \longrightarrow S $, les éléments correspondant dans $Gr^{B_{dR}}_G(C(x)) = \bigsqcup_{\mu \in X_{*}(T)^+} G(B^{+}_{dR}(C)) \mu(\xi)^{-1} G(B^{+}_{dR}(C)) $ appartiennent à $ G(B^{+}_{dR}(C)) \mu(\xi)^{-1} G(B^{+}_{dR}(C)) $, resp. à $\bigsqcup_{\mu' \leq \mu} G(B^{+}_{dR}(C)) \mu'(\xi)^{-1} G(B^{+}_{dR}(C)) $.   
\end{definition}

On dispose alors du théorème suivant.

\begin{théorème} \phantomsection \label{itm : Grass}
	(\cite{SW17})
	
	Pour tout $\mu \in X_{*}(T)^+$, $Gr^{B_{dR}}_{G, \leq \mu}$ est un diamant spatial. Le sous-foncteur $Gr^{B_{dR}}_{G, \leq \mu} \subset Gr^{B_{dR}}_G$ est un sous-foncteur fermé qui est propre sur $\Spa(\breve{E})^{\diamond}$ et $Gr^{B_{dR}}_{G, \mu} \subset Gr^{B_{dR}}_{G, \leq \mu}$ est un sous-foncteur ouvert. De plus le morphisme $ Gr^{B_{dR}}_{G, \leq \mu} \longrightarrow \Spa(\breve{E})^{\diamond} $ est de dimension de transcendance finie (i.e $\dimtrg < \infty $). 
\end{théorème}
\begin{proof}
	Toutes les assertions, sauf la dernière, sont démontrées dans la proposition 19.2.3 et le théorème 19.2.4 de \cite{SW17}. L'assertion sur la finitude de la dimension de transcendance est démontrée implicitement dans les lemmes 19.3.2 et 19.3.3.
	
	D'après le lemme 19.1.5 de loc.cit, toute injection de groupes réductifs $ G \hookrightarrow GL_n $ induit une immersion fermée $Gr^{B_{dR}}_{G, \leq \mu} \hookrightarrow Gr^{B_{dR}}_{GL_n, \leq \mu}$. On peut donc supposer $G = GL_n$.
	 
	Puisqu'il y a un nombre fini de $\mu' \leq \mu$ et que chaque point $x \in | Gr^{B_{dR}}_{GL_n, \leq \mu} |$ appartient à $| Gr^{B_{dR}}_{GL_n, \mu'} |$ pour un $\mu' \leq \mu$, il suffit de montrer la dernière assertion pour $Gr^{B_{dR}}_{GL_n, \mu}$.
	
	En vertu du lemme 19.3.3 de loc.cit., il suffit de considérer la résolution de Demazure $\widetilde{Gr}_{GL_n, \mu}$ de $Gr^{B_{dR}}_{GL_n, \mu}$ (définition 19.3.1 de loc.cit).	Nous allons suivre la suite de réductions du lemme 19.3.2. Le lemme 21.3 de \cite{S17} nous permet de supposer que $\mu$ est minuscule. Ensuite, la question 21.4\footnote{Nous utilisons la définition au début du page 120 et le lemme 21.3 ainsi que la question 21.4 sont donc vérifiés.} de loc.cit nous permet de passer au recouvrement ouvert et comme dans le lemme 19.3.2 de \cite{SW17}, il suffit de considérer le morphisme $Grass(d,n)^{\diamond} \longrightarrow \Spa(\breve{E})^{\diamond}$. Dans ce cas, le résultat découle de \cite{Hub96}. 
\end{proof}
\section{L'espace de modules de Shtukas}
Supposons maintenant que le $G$-fibré $\mathcal{E}$ est trivialisé sur $B^{+}_{dR}(R^{\sharp})$. En terme du recollement de Beauville-Laszlo, $\mathcal{E}$ correspond à un triplet $\Big( \mathcal{E}_e, \mathcal{E}^{+}_{1, B_{dR}}, \iota \Big)$ où $\mathcal{E}^{+}_{1, B_{dR}}$ est le $B^{+}_{dR}$-réseau trivial dans $\mathcal{E}_{1, B_{dR}(R^{\sharp})}$ et où $\iota$ est un isomorphisme
\begin{equation} \phantomsection \label{itm : 1}
\iota : \mathcal{E}_e \otimes_{B_e(R)} B_{dR}(R^{\sharp}) \xrightarrow{\sim} \mathcal{E}^{+}_{1, B_{dR}} \otimes_{B^{+}_{dR}(R^{\sharp})} B_{dR}(R^{\sharp}) \ \simeq \ \mathcal{E}_{1, B_{dR}}.	
\end{equation}

Autrement dit, un fibré $\mathcal{E}$ trivialisé sur $B^{+}_{dR}(R^{\sharp})$ est déterminé par un couple $\Big( \mathcal{E}_e, \iota \Big)$ avec $\iota$ comme décrit dans (\ref{itm : 1}). 
Soit $R^{\sharp} = C$ un corps perfectoïde algébriquement clos ainsi que $R = C^b$. Considérons une modification
\[
\alpha : \mathcal{E}_{b_1}|_{X_{C^b} \setminus D_C} \xrightarrow{\sim} \mathcal{E}_{b_2}|_{X_{C^b} \setminus D_C}
\]
où $b_1, b_2 \in G(\breve{\Q}_p)$.

Puisque $\mathcal{E}_{b_1}$ et $\mathcal{E}_{b_2}$ possèdent une trivialisation naturelle sur $B^{+}_{dR}(C)$, la modification $\alpha$ induit un automorphisme $ g = \iota_{b_2} \circ  \Big( \alpha \otimes \id_{B_{dR}(C)} \Big) \circ \iota_{b_1}^{-1} $ de $\mathcal{E}_{1, B_{dR}}$. D'après la décomposition de Cartan, il existe un unique $ \mu \in X_{*}(T)^+ $ de sorte que $ g \in G(B^{+}_{dR}(C)) \mu(t)^{-1} G(B^{+}_{dR}(C)) $. On dit alors que la modification $\alpha$ est de type $\mu$. \\

Revenons à la situation générale où $R$ est une algèbre affinoïde perfectoïde et $R^{\sharp}$ un débasculement de $R$. Une modification
\[
\alpha : \mathcal{E} |_{X_S \setminus D_{S^{\sharp}}} \xrightarrow{\sim} \mathcal{E}' |_{X_S \setminus D_{S^{\sharp}}}
\]
est dite de type $\mu$ si et seulement si pour tout point géométrique $ x_{C^b} : C^b \longrightarrow \Spa(R, R^{+}) $, la modification $ x^{*}_{C^b} \alpha $ est de type $\mu$.
\begin{definition}
	Supposons nous donné un triplet $(G, \mu, b)$ où $G$ est un groupe réductif sur $\Q_p$, $ b \in G(\breve{\Q}_p) $ et $\mu \in X^{+}_{*}(T)$ avec $b \in B(G, \mu)$. Notons $E$ le corps de définition de $\mu$. L'espace de modules de Shtukas $\Sht(G, \mu, b) \longrightarrow \Spa(\breve{E})^{\diamond}$ est le préfaisceau sur $\Perf_{\overline{\F}_p}$ qui associe à chaque $\breve{E}_p$-espace perfectoïde $S^{\sharp}$, avec un débasculement $(S^{\sharp})^b = S$, l'ensemble des modifications
	\[
	\alpha : \mathcal{E}_b |_{X_{S} \setminus D_{S^{\sharp}}} \xrightarrow{\sim} \mathcal{E}_1 |_{X_{S} \setminus D_{S^{\sharp}}}
	\]
	qui est de type $\mu'$ plus petit que $\mu$.
\end{definition}
Le foncteur $\Sht(G, \mu, b)$ admet une action de $G(\Q_p)$, respectivement de $J_b(\Q_p)$ via $ \alpha \longmapsto g \circ \alpha $ pour $g \in G(\Q_p)$, respectivement $ \alpha \longmapsto \alpha \circ h^{-1} $ pour $h \in J_b(\Q_p)$.

Puisque $b^{\sigma} = b^{-1} \cdot b \cdot b^{\sigma}$ alors $ [b] = [b^{\sigma}] $ dans $B(G)$. Il y a donc un isomorphisme canonique
\[
\widetilde{\sigma} : \mathcal{E}_{b^{\sigma}} \xrightarrow{\sim} \mathcal{E}_{b}.
\]

On définit alors la donnée de descente de $\Sht(G, \mu, b)$ comme l'isomorphisme
\begin{align*}
\widetilde{\mathop{\mathrm{Fr}}\nolimits} : \Sht(G, \mu, b) & \xrightarrow{\sim} \Sht(G, b^{\sigma}, \mu) \\
\alpha & \longmapsto \alpha \circ \widetilde{\sigma}.
\end{align*}

Les actions de $G(\Q_p)$ et $J_b(\Q_p)$ commutent avec la donnée de descente. 
\begin{théorème} \phantomsection \label{itm : Shtukas} \cite{SW17} Étant donné un triplet $(G, b, \mu)$ comme ci-dessus alors $\Sht(G, \mu, b)$ est un diamant localement spatial. De plus, le morphisme $ f : \Sht(G, \mu, b) \longrightarrow \Spa(\breve{E})^{\diamond} $ est partiellement propre et $ \dimtrg f < \infty $. 
\end{théorème}
\begin{proof}
	On a en effet $ \displaystyle \Sht(G, \mu, b) = \mathop{\mathrm{lim}}_{\overleftarrow{K}} \Sht(G, \mu, b)_K$, le théorème 23.1.3 de \cite{SW17} implique alors que $\Sht(G, \mu, b)$ est un diamant localement spatial. D'autre part, la proposition 23.2.1 de loc.cit couplée avec \ref{itm : Grass} implique que $f$ est partiellement propre et $ \dimtrg f < \infty $.
\end{proof}

La proposition suivante est bien connue des experts.

\begin{proposition} \phantomsection \label{itm : modification de type 0}
	Soit $G$ un groupe réductif connexe défini sur $\Q_p$ et $b \in G(\breve{\Q}_p)$. Soient $S = \Spa(R, R^+)$ un espace affinoïde perfectoïde sur $\Spa(\overline{\F}_p)$ ainsi que $\Spa(R^{\sharp}, (R^{\sharp})^+)$ un débasculement. Si $ \alpha : \mathcal{E}_{b | X_{S} \setminus D_{S^{\sharp}} } \xrightarrow{\sim} \mathcal{E}_{b | X_{S} \setminus D_{S^{\sharp}} }$ est une modification de $G$-fibrés de type $\mu = 0$ alors $\alpha$ s'étend en un isomorphisme de $G$-fibrés. En particulier on a une identification $ \Sht(G, \Id, 0) \simeq \underline{G(\Q_p)}. $ 
\end{proposition}
\begin{proof}
	D'après le formalisme tanakien, il suffit de démontrer le résultat pour $G = GL_n$.
	
	Supposons tout d'abord que $\Spa(R, R^0) = \Spa(C^b, \mathcal{O}_{C^b})$ où $C$ est un corps complet algébriquement clos. En utilisant le recollement de Beauville-Laszlo, on peut exprimer $ \mathcal{E}_b $ sous forme 
	\[
	\Big( (B_{C^b}[t^{-1}] \otimes_{\breve{\Q}_p} \breve{\Q}^n_p)^{\varphi \otimes \Id b\sigma} , \big( B^{+}_{dR}(C) \big)^n \Big).
	\]
	
	Comme la modification $\alpha$ est de type $0$, on en déduit que $ g := \big( \alpha \otimes \Id_{B_{dR}(C)} \big) $ est donné par un élément dans $GL_n(B^{+}_{dR}(C))$. En particulier le couple $(\alpha, g)$ est un isomorphisme de $\mathcal{E}_b$.
	
	Le cas particulier où $R = C^b$ couplé avec le lemme 3.4.6 de \cite{CS} implique le cas général où $\Spa(R, R^+)$ est un espace affinoïde perfectoïde.
	
	Enfin, on voit que $\Sht(G, \Id, 0)$ classifie les automorphismes du $G$-fibré trivial $\mathcal{E}_1$. D'après l'exemple 2.21 de \cite{Far16}, on a une identification de diamants $\Sht(G, \Id, 0) = \underline{G(\Q_p)}$. 
	
\end{proof}
\section{Torsions des modifications des $G$-fibrés}

Soient $X$ un schéma et $G$ un $X$-schéma en groupes. Un $G$-torseur est un $\Q_p$-morphisme fidèlement plat $\mathcal{T} \longrightarrow X$ muni d'une action $ G \times \mathcal{T} \longrightarrow \mathcal{T} $ sur l'action triviale de $X$ de sorte que, fppf-localement sur $X$, on a un isomorphisme $G$-équivariant $ \mathcal{T} \simeq G \times X $.
\begin{lemme} (lemme 4.6.1 de \cite{KW}) \phantomsection \label{itm : G-tor G-fib}
	Soit $G \longrightarrow X$ un schéma en groupes réductifs. La catégorie des $G$-fibrés sur $X$ est équivalente à celle des $G$-torseurs sur $X$.
\end{lemme}

Soit $S = \Spa(R, R^+)$ un espace affinoïde perfectoïde sur $\Spa(\F_q)$. Considérons maintenant la courbe de Fargues-Fontaine $ X_S $. Pour chaque $b \in G(\breve{\Q}_p)$, on a un $G$-fibré $\mathcal{E}_b$ sur $X_S$. Notons $\mathcal{T}_b$ le $G$-torseur correspondant à $\mathcal{E}_b$ via le lemme \ref{itm : G-tor G-fib}. On peut décrire ce $G$-torseur par la formule
\begin{equation} \phantomsection \label{itm : réaliser torseur}
\mathcal{T}_b = \Proj \bigoplus_{d \geq 0} \Big( H^0 (Y_S, \mathcal{O}_{Y_S}) \otimes_{\breve{\Q}_p} \breve{\Q}_p[G] \Big)^{\varphi_S \otimes b\sigma = \pi^d}	
\end{equation}
où $\breve{\Q}_p[G]$ est l'algèbre de définition de $G_{\breve{\Q}_p}$. L'action de $G$ sur $\mathcal{T}_b$ est donnée par celle de $G$ sur lui même par translation à  droite. Désormais, notons $\lambda_{g}$ et $\rho_g$ respectivement la multiplication à gauche par $g$ et à droite par $g^{-1}$.

	Étant donné un groupe réductif connexe $G$ défini sur $\Q_p$ et $ \iota : Z_G^0 \hookrightarrow G $ la composante connexe neutre du centre $Z_G$. On va construire un morphisme naturel 
	\[
	\Bun_G \times \Bun_{Z^0_G} \longrightarrow \Bun_G.
	\]
	
	D'après le lemme \ref{itm : G-tor G-fib}, it suffit de définir le morphisme au niveau des torseurs. Considérons $\mathcal{T}$ un $G$-torseur ainsi que $\mathcal{T}'$ un $Z_G^0$-torseur sur $X$. Puisque $Z_G^0$ est contenu dans le centre de $G$, on peut munir $\mathcal{T}$ d'une action de $Z_G^0$ telle que l'action de $G$ et de $Z_G^0$ commutent. Plus précisément, un élément $g \in Z_G^0$ agit par action de $g$ sur $\mathcal{T}$. Le produit contracté $ \mathcal{T} \times_{Z_G^0} \mathcal{T}'$ est alors un $G$-torseur. 
\begin{proposition} \phantomsection \label{itm : construction fondamentale}	
	 Pour $b \in B(G)$ et $h \in B(Z_G^0)$, l'image de $(\mathcal{E}_b, \mathcal{E}_h)$ via le morphisme ci-dessus est $\mathcal{E}_{b \overline{h}}$ où $\overline{h} = \iota(h)$.	
\end{proposition}
\begin{proof}
	Il suffit de démontrer $ \mathcal{T}_b \times_{Z_G^0} \mathcal{T}_h = \mathcal{T}_{b\overline{h}} $ pour $b \in B(G)$ et $h \in B(Z_G^0)$.
	
	D'après la formule (\ref{itm : réaliser torseur}) on a
	\[
	\mathcal{T}_b = \Proj \bigoplus_{d \geq 0} \Big( H^0 (Y_S, \mathcal{O}_{Y_S}) \otimes_{\breve{\Q}_p} \breve{\Q}_p[G] \Big)^{\varphi_S \otimes b\sigma = \pi^d}
	\]
	où l'action de $G$ sur $\mathcal{T}_b$ est donnée par action de $G$ sur lui même par translation à  droite et action de $Z_G^0$ est donnée par translation à gauche.
	
	On a le diagramme commutatif suivant où $\mu$ désigne la multiplication de $G$.
	\begin{center} \phantomsection \label{itm : diagramme}
		\begin{tikzpicture}[scale = 1]
		\draw (-1,0) node {$G_{\breve{\Q}_p} \times (Z_G^0)_{\breve{\Q}_p} $};
		\draw (3,0) node {$ G_{\breve{\Q}_p} $};
		\draw (-1,-1.5) node {$ G_{\breve{\Q}_p} \times (Z_G^0)_{\breve{\Q}_p}  $};
		\draw (3,-1.5) node {$G_{\breve{\Q}_p}$};
		\draw [->] (-1,-0.4) -- (-1,-1.15) node[midway, left]{$\lambda_b  \times \rho_{\overline{h}^{-1}}$};
		\draw [->] (3, -0.4) -- (3,-1.15)node[midway, right]{$\lambda_{b\overline{h}}$};
		\draw [->] (0,0) -- (2.5,0) node [midway, above]{$\mu$};
		\draw [->] (0, -1.5) -- (2.5, -1.5) node [midway, above]{$\mu$};
		\end{tikzpicture}
	\end{center}
	
	On en déduit que la co-multiplication $\mu : \breve{\Q}_p[G] \longrightarrow \breve{\Q}_p[G] \otimes_{\breve{\Q}_p} \breve{\Q}_p[Z_G^0]$ induit l'isomorphisme voulu $ \mathcal{T}_b \times_{Z_G^0} \mathcal{T}_h = \mathcal{T}_{b\overline{h}} $. 
\end{proof}
Pour $T$ un tore défini sur $\breve{\Q}_p$, d'après Kottwitz on a une bijection
\[
\kappa : B(T)_{\text{basic}} = B(T) \xrightarrow{\sim} X_*(T)_{\Gamma}.
\]

On note $b_{\lambda}$ l'élément correspondant à un $\lambda \in X_*(T)_{\Gamma}$ via cette bijection. On voit également que $ B(T, \lambda) = \{ b_{\lambda} \} $, en particulier pour $C$ un corps algébriquement clos, il existe une modification de type $\lambda$ 
\[
\alpha : \mathcal{E}_{b | X_{C^b} \setminus D_C } \xrightarrow{\sim} \mathcal{E}_{1 | X_{C^b} \setminus D_C }
\] 
si et seulement si $\mathcal{E}_{b} \simeq \mathcal{E}_{b_{\lambda}}$.

Étant données deux modifications $\alpha, \alpha'$ de type $ \mu = 0 \in X_*(T)$, d'après la proposition \ref{itm : construction fondamentale}, on peut construire une autre modification qui est aussi de type $ \mu = 0$
\[
\mathcal{E}_1 = \mathcal{E}_1 \times_{T} \mathcal{E}_1 \xrightarrow{\alpha \times \alpha'} \mathcal{E}_1 \times_{T} \mathcal{E}_1 = \mathcal{E}_1.
\]

En utilisant cette construction, on peut munir l'espace de Shtukas $\Sht(T, 0)$ d'une structure de diamant en groupes. En effet, pour $ S = \Spa(R, R^+)$ un espace affinoïde perfectoïde sur $\Spa(\overline{\F}_p)$ ainsi que $S^{\sharp} = \Spa(R^{\sharp}, (R^{\sharp})^+)$ un débasculement. Étant données deux modifications de type $ \mu = 0$
\[
\alpha, \beta : \mathcal{E}_{1 | X_{S} \setminus D_{S^{\sharp}} } \xrightarrow{\sim} \mathcal{E}_{1 | X_{S} \setminus D_{S^{\sharp}} },
\] 
on définit le produit de $\alpha$ et $\beta$ par la formule 
\[
\alpha \times \beta : \mathcal{E}_1 = \mathcal{E}_1 \times_{T} \mathcal{E}_{1 | X_{S} \setminus D_{S^{\sharp}} } \xrightarrow{\alpha \times \beta} \mathcal{E}_1 \times_{T} \mathcal{E}_{1 | X_{S} \setminus D_{S^{\sharp}} } = \mathcal{E}_1.
\]

Alors $\Sht(T, 0)$ avec ce produit est un diamant en groupes. De plus, d'après la proposition \ref{itm : modification de type 0}, on a une identification de diamants en groupes $\Sht(T, 0) \simeq \underline{T(\Q_p)}$.





\begin{lemme} \phantomsection \label{itm : associativité}
	Pour une modification $\alpha : \mathcal{E}_{| X_{S} \setminus D_{S^{\sharp}}} \longrightarrow \mathcal{E}_{1 | X_{S} \setminus D_{S^{\sharp}}}$ de $G$-torseurs et deux modifications $ \beta, \gamma : \mathcal{E}_{| X_{S} \setminus D_{S^{\sharp}}} \longrightarrow \mathcal{E}_{1 | X_{S} \setminus D_{S^{\sharp}}}$ de $Z_G^0$-torseurs on a $\big( \alpha \times \beta \big) \times \gamma = \alpha \times \big( \beta \times \gamma \big)$ comme modifications de $G$-torseurs. 
\end{lemme}
\begin{proof}
	Il s'agit d'utiliser le lemme \ref{itm : G-tor G-fib} pour calculer explicitement les isomorphismes des torseurs.
\end{proof}

Étant donné un groupe réductif connexe $G$ défini sur $\Q_p$ et $\iota : Z_G^0 \hookrightarrow G$ la composante connexe du tore central. Le morphisme $\iota$ induit un morphisme $ \theta : X_*(Z_G^0) \longrightarrow X_*(T)^+ $ où $T$ est le tore maximal de $G$. 
Afin d'alléger les notations, pour $\lambda \in X_*(Z_G^0)$, on note encore $\lambda$ son image par $\theta$ dans $X_*(T)^+$.

Soient $ S = \Spa(R, R^+)$ un espace affinoïde perfectoïde sur $\Spa(\overline{\F}_p)$ ainsi que $S^{\sharp} = \Spa (R^{\sharp}, (R^{\sharp})^+)$ un débasculement. Si l'on dispose d'une modification $ \alpha : \mathcal{E}_{h | X_{S} \setminus D_{S^{\sharp}}} \longrightarrow \mathcal{E}_{1 | X_{S} \setminus D_{S^{\sharp}}} $ de $Z_G^0$-fibrés de type $\lambda$, alors pour tout $G$-fibré $\mathcal{E}_b$ et tout isomorphisme $f : \mathcal{E}_b \xrightarrow{\sim} \mathcal{E}_b$ on a une modification de $G$-fibrés
\[
\overline{\alpha} : \mathcal{E}_{b\overline{h}} \simeq \mathcal{E}_b \times_{Z_G^0} \mathcal{E}_{h | X_{S} \setminus D_{S^{\sharp}}} \xrightarrow{f \times \alpha} \mathcal{E}_b \times_{Z_G^0} \mathcal{E}_{1 | X_{S} \setminus D_{S^{\sharp}}} \simeq \mathcal{E}_b.
\]

De même, pour une modification de $G$-fibrés $\beta : \mathcal{E}_b \longrightarrow \mathcal{E}_1$ de type $\mu$, on peut construire une modifications de $G$-fibrés 
\[
\overline{\beta} : \mathcal{E}_b = \mathcal{E}_b \times_{Z_G^0} \mathcal{E}_{1 | X_{S} \setminus D_{S^{\sharp}}} \xrightarrow{\beta \times \Id} \mathcal{E}_1 \times_{Z_G^0} \mathcal{E}_{1 | X_{S} \setminus D_{S^{\sharp}}} = \mathcal{E}_1.
\]
\begin{proposition} \phantomsection \label{itm : type} Supposons que $G$ est déployé sur $\breve{\Q}_p$.
	\begin{enumerate}
		\item[(i)] La modification $\overline{\alpha} : \mathcal{E}_{b\overline{h}} \longrightarrow \mathcal{E}_b $ est de type $\lambda$.
		\item[(ii)] La modification $\overline{\beta}$ est de type $\mu$.
	\end{enumerate}
\end{proposition}
\begin{proof}
	D'après le formalisme tanakien, il suffit de démontrer le résultat pour toutes les représentations $ \rho : G \longrightarrow GL(V)$ de $G$. Puisque l'on est dans le cas de caractéristique $0$ et que $G$ est réductif, la catégorie $\Rep_{\Q_p}G$ est semi-simple. Il suffit donc de traiter les représentations irréductibles. Lorsque $ \rho : G \longrightarrow GL(V) $ est irréductible, on a $\rho (Z_G^0) \subset Z_{GL(V)}$. On obtient également un cocaractère de $Z_{GL(V)}$
	\[
	\rho \circ \lambda : \G_{m, \overline{\Q}_p} \longrightarrow Z_{GL(V), \overline{\Q}_p}.
	\]
	
	On peut supposer que $G = GL(V)$. On a alors $Z_G^0 = GL_1$ et $ X_*(Z_G^0) \simeq \Z $, de plus le cocaractère $ \lambda$ correspond à un $d \in \Z$. On en déduit que $h = [p^d] \in B(Z_G^0)$ puisque la modification $\alpha : \mathcal{E}_{h | X_{S} \setminus D_{S^{\sharp}}} \longrightarrow \mathcal{E}_{1 | X_{S} \setminus D_{S^{\sharp}}}$ est de type $\lambda$.
	
	Pour calculer le type de $\overline{\alpha}$, il suffit de considérer la cas $R = C^b $ où $C$ est un corps complet algébriquement clos. D'après le recollement de Beauville-Laszlo, on a la description suivante : 
	\[
	\mathcal{E}_{p^d} = \Big( (B_{C^b}[t^{-1}] \otimes_{\breve{\Q}_p} \breve{\Q}_p)^{\varphi \otimes p^d\sigma} , B^{+}_{dR}(C) \Big) \quad \quad \quad \mathcal{E}_{1} = \Big( (B_{C^b}[t^{-1}] \otimes_{\breve{\Q}_p} \breve{\Q}_p)^{\varphi \otimes \Id\sigma} , B^{+}_{dR}(C) \Big).
	\]
	
	La modification $\alpha$ de type $\lambda$ est alors donnée par l'isomorphisme 
	\begin{align*}
	(B_{C^b}[t^{-1}] \otimes_{\breve{\Q}_p} \breve{\Q}_p)^{\varphi \otimes p^d\sigma} &\longrightarrow (B_{C^b}[t^{-1}] \otimes_{\breve{\Q}_p} \breve{\Q}_p)^{\varphi \otimes \Id\sigma} \\
	x &\longmapsto t^{-d}_{C^b}x. 	
	\end{align*}
	
	D'après la formule (\ref{itm : réaliser torseur}), la modification $\alpha$ s'écrit en termes  de $Z_G^0$-torseurs
	\begin{align*}  
	\mathcal{T}_{p^d  | X_{C^b} \setminus D_{C}}  \simeq \Big(B_{C^b}[t^{-1}]  \otimes_{\breve{\Q}_p} \breve{\Q}_p[Z_G^0] \Big)^{\varphi \otimes p^d \sigma} &\xrightarrow{\sim} \Big(B_{C^b}[t^{-1}]  \otimes_{\breve{\Q}_p} \breve{\Q}_p[Z_G^0] \Big)^{\varphi \otimes \Id \sigma} \simeq \mathcal{T}_{1  | X_{C^b} \setminus D_{C}} \\
	x \otimes x' &\longmapsto t^{-d}x \otimes x'.	
	\end{align*}
	
	En utilisant le diagramme \ref{itm : diagramme}, la modification $\overline{\alpha}$ s'écrit en termes de $G$-torseurs
	\begin{align*}  
	\mathcal{T}_{p^db  | X_{C^b} \setminus D_{C}}  \simeq \Big(B_{C^b}[t^{-1}]  \otimes_{\breve{\Q}_p} \breve{\Q}_p[G] \Big)^{\varphi \otimes p^db \sigma} &\xrightarrow{\sim} \Big(B_{C^b}[t^{-1}]  \otimes_{\breve{\Q}_p} \breve{\Q}_p[G] \Big)^{\varphi \otimes \Id \sigma} \simeq \mathcal{T}_{1  | X_{C^b} \setminus D_{C}} \\
	x \otimes x' &\longmapsto t^{-d}_{C^b}x \otimes x'.	
	\end{align*}
	
	Finalement la modification $\overline{\alpha}$ s'écrit sous la forme
	\begin{align*}
	(B_{C^b}[t^{-1}] \otimes_{\Q_p} V)^{\varphi \otimes p^db\sigma} &\longrightarrow (B_{C^b}[t^{-1}] \otimes_{\Q_p} V)^{\varphi \otimes \Id\sigma} \\
	x &\longmapsto t^{-d}x. 	
	\end{align*}
	
	Il est alors aisé de voir que la modification $\overline{\alpha}$ est de type $(d,\cdots, d) \in \Z^{\dim_{\Q_p}V} = X_*(GL(V))$, autrement dit $\overline{\alpha}$ est de type $\lambda$, ce qui démontre le point $(i)$. 
	
	On démontre le point (ii) par le même argument.
\end{proof}

 On peut ainsi utiliser la proposition \ref{itm : type} pour définir une action de $\Sht(Z_G^0, 0) \simeq \underline{Z_G^0(\Q_p)}$ sur $\Sht(G, \mu, b)$ et sur $\Sht(Z_G^0, \lambda)$. Soient $ S = \Spa(R, R^+)$ un espace affinoïde perfectoïde sur $\Spa(\overline{\F}_p)$ ainsi que $ S^{\sharp} = \Spa(R^{\sharp}, (R^{\sharp})^+)$ un débasculement. Étant donnée une modification de type $ \mu = 0$ de $Z_G^0$-fibrés $ \alpha : \mathcal{E}_{1 | X_{S} \setminus D_{S^{\sharp}} } \xrightarrow{\sim} \mathcal{E}_{1 | X_{S} \setminus D_{S^{\sharp}} }$ et une modification de type $\mu$ de $G$-fibrés $ \alpha' : \mathcal{E}_{b | X_{S} \setminus D_{S^{\sharp}} } \xrightarrow{\sim} \mathcal{E}_{1 | X_{S} \setminus D_{S^{\sharp}} } $, on définit l'action de $\alpha$ sur $\alpha'$ par la formule
\[
\mathcal{E}_b = \mathcal{E}_b \times_{Z_G^0} \mathcal{E}_{1 | X_{S} \setminus D_{S^{\sharp}} } \xrightarrow{\alpha' \times \alpha} \mathcal{E}_1 \times_{Z_G^0} \mathcal{E}_{1 | X_{S} \setminus D_{S^{\sharp}} } = \mathcal{E}_1.
\] 

D'après la proposition \ref{itm : type}, $\alpha' \times \alpha$ est une modification de type $\mu$. Le lemme \ref{itm : associativité} implique que cela définit une action de $\underline{Z_G^0(\Q_p)}$ sur $\Sht(G, \mu, b)$.

De même, si l'on a une modification de type $\lambda$ de $Z_G^0$-fibrés $ \alpha' : \mathcal{E}_{b_{\lambda} | X_{S} \setminus D_{S^{\sharp}} } \xrightarrow{\sim} \mathcal{E}_{1 | X_{S} \setminus D_{S^{\sharp}} } $, on peut définir une autre modification de type $\lambda$ par la formule
\[
\mathcal{E}_{b_{\lambda}} = \mathcal{E}_1 \times_{Z_G^0} \mathcal{E}_{b_{\lambda} | X_{S} \setminus D_{S^{\sharp}}} \xrightarrow{ \alpha^{-1} \times \alpha' } \mathcal{E}_{1} \times_{Z_G^0} \mathcal{E}_{1 | X_{S} \setminus D_{S^{\sharp}}} = \mathcal{E}_{1} 
\]
où les actions de $Z_G^0$ sur $ \mathcal{E}_1 \times_{Z_G^0} \mathcal{E}_{b_{\lambda}}$ et $\mathcal{E}_{1} \times_{Z_G^0} \mathcal{E}_1$ sont induites par celle sur la première composante.

\begin{remarque} \phantomsection \label{itm : action de Z_G}
	Étant donné un élément $ g \in Z_G^0(\Q_p) $ correspondant à une modification de type $ \mu = 0$ de $Z_G^0$-fibrés $ \alpha : \mathcal{E}_{1 | X_{S} \setminus D_{S^{\sharp}} } \xrightarrow{\sim} \mathcal{E}_{1 | X_{S} \setminus D_{S^{\sharp}} }$ alors action de $\alpha$ est celle de $\iota(g)$ où $\iota : Z_G^0(\Q_p) \hookrightarrow G(\Q_p)$ est l'injection canonique. En effet, on a $ \alpha' \times \alpha = \Big( \alpha' \circ \Id_{\mathcal{E}_{1}} \Big) \times \Big( \Id_{ \overline{\mathcal{E}_1}} \circ \alpha \Big) = \Big( \alpha' \times \Id_{ \overline{\mathcal{E}_1}} \Big) \circ \Big( \Id_{\mathcal{E}_{1}} \times \alpha \Big) = \alpha' \circ \iota(g) $.
\end{remarque}

\begin{lemme} \phantomsection \label{itm : transitivement}
	Soient $T$ un tore défini sur $\Q_p$ et $\alpha, \beta : \mathcal{E}_{\lambda} \longrightarrow \mathcal{E}_1$ deux modifications de $T$-fibrés de type $\lambda$. Alors il existe une modification $ \gamma : \mathcal{E}_1 \xrightarrow{\sim} \mathcal{E}_1 $ (de type $0$) de sorte que $\beta = \alpha \times \gamma$. 
\end{lemme}
\begin{proof}
	Soit $ \alpha' : \mathcal{E}_{\lambda^{-1}} \longrightarrow \mathcal{E}_1 $ une modification de $T$-fibrés de type $\lambda^{-1}$ (une telle modification existe). Considérons la modification suivante 
	\[
	\alpha \times \alpha' : \mathcal{E}_{\lambda} \times_{T} \mathcal{E}_{\lambda^{-1} | X_{S} \setminus D_{S^{\sharp}}} \xrightarrow{\alpha \times \alpha'} \mathcal{E}_{1} \times_{T} \mathcal{E}_{1 | X_{S} \setminus D_{S^{\sharp}}}
	\]
	
	Puisque $\mathcal{E}_{\lambda} \times_{T} \mathcal{E}_{\lambda^{-1}} = \mathcal{E}_1$ (d'après la proposition \ref{itm : construction fondamentale}), la modification $\alpha \times \alpha'$ est en fait une modification de type $0$ de $\mathcal{E}_1$. 
Posons $ \alpha' \times \big( \alpha \times \alpha' \big)^{-1} \times \beta $ le produit contracté
	\[
	\mathcal{E}_{\lambda^{-1}} \times_{T} \mathcal{E}_1 \times_{T} \mathcal{E}_{\lambda} \xrightarrow{\alpha' \times \big( \alpha \times \alpha' \big)^{-1} \times \beta} \mathcal{E}_1 \times_{T} \mathcal{E}_1 \times_{T} \mathcal{E}_1.
	\]
	
	En particulier d'après la proposition \ref{itm : construction fondamentale}, $\mathcal{E}_{\lambda^{-1}} \times_{T} \mathcal{E}_1 \times_{T} \mathcal{E}_{\lambda} = \mathcal{E}_1$ et on voit que $\alpha' \times \big( \alpha \times \alpha' \big)^{-1} \times \beta$ est une modification de type $0$ de $\mathcal{E}_1$. Finalement, d'après le lemme \ref{itm : associativité}, on a 
	\begin{equation*}
	\alpha \times \Big( \alpha' \times \big( \alpha \times \alpha' \big)^{-1} \times \beta \Big) = \Big( \Id \times \beta \Big) = \beta
	\end{equation*}
	ce qui termine la démonstration.
\end{proof}

\section{Modifications centrales des espaces de Shtukas}
Dans cette section, on suppose que $G$ est déployé sur $\Q_p$. En utilisant la proposition \ref{itm : type}, on prouve un énoncé géométrique reliant deux espaces de modules de Shtukas. On donnera également une relation cohomologique de ces espaces.

Puisque le produit fibré de diamants existe (\cite{S17}, Prop. 11.4), alors $\Sht(G, \mu, b) \times_{\Spa(\breve{\Q}_p)^{\diamond}} \Sht(Z_G^0, \lambda)$ est un diamant. C'est le faisceau sur $\Perf_{\breve{\Q}_p, \text{pro-ét}}$ qui associe à chaque $\breve{\Q}_p$-espace perfectoïde $S^{\sharp}$, avec $(S^{\sharp})^b = S$, l'ensemble des couples $(\alpha, \beta)$ où $\alpha : \mathcal{E}_{b | X_{S} \setminus D_{S^{\sharp}}} \xrightarrow{\sim} \mathcal{E}_{1 | X_{S} \setminus D_{S^{\sharp}}}$ est une modification de $G$-fibrés de type $\mu'$ plus petit que $\mu$ et $ \beta : \mathcal{E}_{b_{\lambda} | X_{S} \setminus D_{S^{\sharp}}} \xrightarrow{\sim} \mathcal{E}_{1 | X_{S} \setminus D_{S^{\sharp}}} $ est une modification de $G$-fibrés de type $\lambda$. Comme le diamant en groupes $\Sht(Z_G^0, 0) \simeq \underline{Z_G^0(\Q_p)}$ agit sur $\Sht(G, \mu, b)$ et $\Sht(Z_G^0, \lambda)$, on peut définir le faisceau quotient.
\begin{definition}
	On définit $\Sht(G, \mu, b) \times_{\underline{Z_G^0(\Q_p)}} \Sht(Z_G^0, \lambda)$ comme le faisceau quotient. Plus précisément, c'est la faisceautisation du pré-faisceau sur $\Perf_{\overline{\F}_p}$ qui associe à chaque $\breve{\Q}_p$-espace perfectoïde $S^{\sharp}$, avec $(S^{\sharp})^b = S$, l'ensemble $\Sht(G, \mu, b) \times_{\Spa(\breve{\Q}_p)^{\diamond}} \Sht(Z_G^0, \lambda)\big(S^{\sharp}, S \big) / \underline{Z_G^0(\Q_p)} \big( S^{\sharp}, S \big) $.
\end{definition}

Le groupe $G(\Q_p) \times \J_b(\Q_p)$ agit sur le diamant $\Sht(G, \mu, b)$, on peut utiliser cette action pour définir une action de ce groupe sur le faisceau $\Sht(G, \mu, b) \times_{\underline{Z_G^0(\Q_p)}} \Sht(Z_G^0, \lambda)$. Un élément $g \in G(\Q_p)$ correspond à un isomorphisme $\mathcal{E}_1 \xrightarrow{\sim} \mathcal{E}_1$ et $g$ agit sur $\Sht(G, \mu, b)$ par composition avec cet isomorphisme. On définit alors l'action de $g$ sur $\Sht(G, \mu, b) \times_{\underline{Z_G^0(\Q_p)}} \Sht(Z_G^0, \lambda)$ par la formule $ g \Big( (\alpha, \beta) \Big) = \Big( g \circ \alpha, \beta \Big) $. Cette action est bien définie car $ (g \circ \alpha) \times \gamma = (g \times \Id_{ \mathcal{E}_1}) \circ (\alpha \times \gamma) = g \circ (\alpha \times \gamma) $ pour tout $\gamma \in \underline{Z_G^0(\Q_p)}$.

De la même manière, on définit une action de $\J_b(\Q_p)$\footnote{ $\overline{b}_{\lambda}$ étant un élément central de $G(\breve{\Q}_p)$, on a alors $ \J_b(\Q_p) = \J_{b \overline{b}_{\lambda}} (\Q_p) $.} sur $\Sht(G, \mu, b) \times_{\underline{Z_G^0(\Q_p)}} \Sht(Z_G^0, \lambda)$. 
Il y a des isomorphismes canoniques
\[
\sigma_b : \mathcal{E}_{b^{\sigma}} \xrightarrow{\sim} \mathcal{E}_{b} \quad \quad \quad \sigma_{b_{\lambda}} : \mathcal{E}_{b_{\lambda}^{\sigma}} \xrightarrow{\sim} \mathcal{E}_{b_{\lambda}}.
\]
On définit alors la donnée de descente de $\Sht(G, \mu, b) \times_{\underline{Z_G^0(\Q_p)}} \Sht(Z_G^0, \lambda)$ comme l'isomorphisme
\begin{align*}
\mathop{\mathrm{Fr}}\nolimits : \Sht(G, b^{\sigma}, \mu) \times_{\Sht(Z_G^0, 0)} \Sht(Z_G^0, \lambda) & \xrightarrow{\sim} \Sht(G, b^{\sigma}, \mu) \times_{\Sht(Z_G^0, 0)} \Sht(Z_G^0, \lambda) \\
(\alpha, \beta) & \longmapsto (\alpha \circ \sigma_b, \beta \circ \sigma_{b_{\lambda}}).
\end{align*}
On vérifie aisément que cela est bien défini.
\begin{lemme} \phantomsection \label{itm :  données de descentes}
	Notons $\sigma_{b\overline{b}_{\lambda}} : \mathcal{E}_{(b\overline{b}_{\lambda})^{\sigma}} \xrightarrow{\sim} \mathcal{E}_{b\overline{b}_{\lambda}}$ l'isomorphisme canonique. On a alors
	\[
	\sigma_b \times \sigma_{\overline{b}_{\lambda}} = \sigma_{b\overline{b}_{\lambda}}.
	\]
\end{lemme}
\begin{proof}
	D'après la formule (\ref{itm : réaliser torseur}) on a
	\[
	\mathcal{T}_b = \Proj \bigoplus_{d \geq 0} \Big( H^0 (Y_S, \mathcal{O}_{Y_S}) \otimes_{\breve{\Q}_p} \breve{\Q}_p[G] \Big)^{\varphi_S \otimes b\sigma = \pi^d},
	\]
	et de même pour le torseur $\mathcal{T}_{b^{\sigma}}$. On en déduit que l'isomorphisme canonique $\sigma_b : \mathcal{T}_{b^{\sigma}} \simeq \mathcal{T}_b$ est donné par $ \Id_S \otimes b $.
	
	D'après la proposition \ref{itm : construction fondamentale}, on sait que la co-multiplication $\mu : \breve{\Q}_p[G] \longrightarrow \breve{\Q}_p[G] \otimes_{\breve{\Q}_p} \breve{\Q}_p[Z_G^0]$ induit l'isomorphisme $ \mathcal{T}_b \times_{Z_G^0} \mathcal{T}_{b_{\lambda}} = \mathcal{T}_{b\overline{b}_{\lambda}} $. On en déduit que $\sigma_b \times \sigma_{\overline{b}_{\lambda}} = \sigma_{b\overline{b}_{\lambda}}$.
\end{proof}
\begin{théorème} \phantomsection \label{itm : géométrique}
	Il y a un isomorphisme $G(\Q_p) \times \J_b(\Q_p)$-équivariant de faisceaux pro-étale qui commute avec les données de descentes.
	\[
	\Sht(G, \mu, b) \times_{\underline{Z_G^0(\Q_p)}} \Sht(Z_G^0, \lambda) \longrightarrow \Sht(G, b\overline{b}_{\lambda} ,\mu \cdot  \lambda).
	\]
\end{théorème}
\begin{proof}
	Pour une $\breve{\Q}_p$-algèbre perfectoïde $R^{\sharp}$ avec $(R^{\sharp})^b = R$ et un couple $(\alpha, \beta) \in \Sht(G, \mu, b) \times_{\Spa(\breve{\Q}_p)^{\diamond}} \Sht(Z_G^0, \lambda)\big(R^{\sharp}, R \big) $, on peut construire, d'après la proposition \ref{itm : construction fondamentale}, une modification 
	\[
	\alpha \times \beta : \mathcal{E}_{b\overline{b}_{\lambda} | X_{S} \setminus D_{S^{\sharp}}} \longrightarrow \mathcal{E}_{1 | X_{S} \setminus D_{S^{\sharp}}}
	\]
	D'autre part, en écrivant $ \alpha = \Id_{\mathcal{E}_b} \circ \alpha $ et $ \beta = \beta \circ \Id_{\mathcal{E}_1}$, on voit que 
	\[
	\alpha \times \beta = \Big( \alpha \times \Id_{\mathcal{E}_1} \Big) \circ \Big( \Id_{\mathcal{E}_b} \times \beta \Big)
	\]
	où $ \Id_{\mathcal{E}_b} \times \beta $ et $ \alpha \times \Id_{\mathcal{E}_1} $ sont des modifications de $G$-fibrés
	\[
	\Id_{\mathcal{E}_b} \times \beta : \mathcal{E}_{b\overline{b}_{\lambda}} \longrightarrow \mathcal{E}_b \quad \quad \quad \alpha \times \Id_{\mathcal{E}_1} : \mathcal{E}_b \longrightarrow \mathcal{E}_1
	\]
	qui sont respectivement de type $\lambda$ et $\mu$ d'après la proposition \ref{itm : type}. 
	
	Afin de calculer le type de $\alpha \times \beta$, on utilise la décomposition de Cartan
	\[
	G(B_{dR}(C)) = \bigsqcup_{\eta \in X_{*}(T)/W} G(B^{+}_{dR}(C)) \eta(\xi)^{-1} G(B^{+}_{dR}(C)) 
	\]
	où $C$ est un corps perfectoïde algébriquement clos et $T$ un tore maximal de $G$.
	
	Pour $g \in G(B^{+}_{dR}(C)) \mu(\xi)^{-1} G(B^{+}_{dR}(C))$ et $ h \in G(B^{+}_{dR}(C)) \lambda(\xi)^{-1} G(B^{+}_{dR}(C)) $, on voit que $ gh \in G(B^{+}_{dR}(C)) \mu \cdot \lambda(\xi)^{-1} G(B^{+}_{dR}(C)) $ car $\lambda(\xi)^{-1}$ est central. Autrement dit la modification $\alpha \times \beta$ est de type $ \mu \cdot \lambda $.
	
	On a construit un morphisme $ \Phi : \Sht(G, \mu, b) \times_{\Spa(\breve{\Q}_p)^{\diamond}} \Sht(Z_G^0, \lambda) \longrightarrow \Sht(G, b\overline{b}_{\lambda} ,\mu \cdot \lambda) $. Montrons ensuite que ce morphisme passe au quotient par $ \Sht(Z_G^0, 0)$.
	
	Soient $(\alpha, \beta) \in \Sht(G, \mu, b) \times_{\Spa(\breve{\Q}_p)^{\diamond}} \Sht(Z_G^0, \lambda)$ et $\gamma \in \underline{Z_G^0(\Q_p)}$, il faut montrer que $ \Big( \alpha \times \gamma \Big) \times \Big( \gamma^{-1} \times \beta \Big) = \alpha \times \beta$. Or d'après le lemme \ref{itm : associativité}, cette identité est vérifiée. \\
	
	On va montrer que $\Phi$ est en fait un isomorphisme. Il s'agit de montrer que $ \Phi_{R^{\sharp}, R} $ est un isomorphisme pour tous les couples $(R^{\sharp}, R)$. On aura besoin de l'ensemble auxiliaire suivant
	\[
	S = \Big\{ (\alpha, \beta, \gamma) \in \Sht(G, \mu \cdot \lambda, b\overline{b}_{\lambda}) \times_{\Spa(\breve{\Q}_p)^{\diamond}} \Sht(Z_G^0, \lambda^{-1}) \times_{\Spa(\breve{\Q}_p)^{\diamond}} \Sht(Z_G^0, \lambda) (R^{\sharp}, R) \Big\} / \sim
	\]
	où $(\alpha, \beta, \gamma) \sim (\alpha', \beta', \gamma')$ si, et seulement si, il existe $\theta \in \Sht(Z_G^0, 0)(R^{\sharp}, R)$ de sorte que $\alpha' = \alpha \times \theta, \ \beta' = \theta^{-1} \times \beta$ ou $ \beta' = \beta \times \theta, \ \gamma' = \theta^{-1} \times \gamma $.
	
	On a donc une application canonique
	\begin{align*}
	f : \quad S \quad & \longrightarrow \Sht(G, \mu \cdot \lambda, b\overline{b}_{\lambda})(R^{\sharp}, R) \\
	(\alpha, \beta, \gamma) & \longmapsto \alpha \times \beta \times \gamma.
	\end{align*}
	
	On montre que cela est une bijection. Étant donné une modification $\alpha \in \Sht(G, \mu \cdot \lambda, b\overline{b}_{\lambda})(R^{\sharp}, R)$, on choisit $\beta$ et $\gamma$ des modifications de $Z_G^0$-fibrés de type $\lambda^{-1}$ et $\lambda$ respectivement. Puisque $\beta \times \gamma \in \underline{Z_G^0(\Q_p)}(R^{\sharp}, R) $ on peut choisir $\delta \in \underline{Z_G^0(\Q_p)}(R^{\sharp}, R)$ de sorte que $ \delta \times \beta \times \alpha = \Id_{ \mathcal{E}_{1 | X_{S} \setminus D_{S^{\sharp}} } } $. On a alors $ f(\alpha, \delta \times \beta, \gamma) = \alpha $. Autrement dit $f$ est une surjection. Le même type d'argument couplé avec le lemme \ref{itm : transitivement} montre également que $f$ est injective, finalement on en déduit que $f$ est une bijection.
	
	Maintenant on définit l'application $\Phi^{-1}_{R^{\sharp}, R}$ par la formule
	\begin{align*}
	\Phi^{-1}_{R^{\sharp}, R} : \quad S \quad & \longrightarrow \Sht(G, \mu, b) \times_{\Spa(\breve{\Q}_p)^{\diamond}} \Sht(Z_G^0, \lambda)(R^{\sharp}, R) \\
	(\alpha, \beta, \gamma) &\longmapsto (\alpha \times \beta, \gamma).
	\end{align*}
	Il est facile de vérifier que $\Phi^{-1}_{R^{\sharp}, R}$ est l'application inverse de $\Phi_{R^{\sharp}, R}$. Reste finalement à voir que $\Phi_{R^{\sharp}, R}$ commute avec toutes les structures. La commutativité avec l'action de $G(\Q_p) \times \J_b(\Q_p)$ est claire et la commutativité avec les données de descentes résulte du lemme \ref{itm :  données de descentes}.
\end{proof}

 On aimerait descendre en niveau fini l'énoncé géométrique de la proposition \ref{itm : géométrique}. Considérons un sous groupe compact $K \subset G(\Q_p)$ et posons $K_Z := K \cap Z_G^0(\Q_p)$ lequel est un sous groupe compact de $Z_G^0(\Q_p)$. Puisque $Z_G^0(\Q_p)$ est commutatif, $ Z_G^0(\Q_p) / K_Z $ est encore un groupe (discret).

\begin{corollaire} \phantomsection \label{itm : niveau fini}
	Il y a un isomorphisme $\J_b(\Q_p)$-équivariant de faisceaux pro-étale qui commute avec les données de descentes.
	\[
	\big( \Sht(G, \mu, b)/K \big) \times_{\underline{Z_G^0(\Q_p)/ K_Z}} \big( \Sht(Z_G^0, \lambda)/ K_Z \big) \longrightarrow  \Sht(G,\mu \cdot \lambda, b\overline{b}_{\lambda} )/K.
	\]
	
	De plus, l'action de $Z^0_G(\Q_p) / K_Z$ sur $\big( \Sht(G, \mu, b)/K \big) \times \big( \Sht(Z_G^0, \lambda)/ K_Z \big)$ est sans point fixe.
\end{corollaire}
\begin{proof}
	D'après la remarque \ref{itm : action de Z_G}, l'action de $g \in \Z^0_G(\Q_p)$ sur $\Sht(G, \mu, b)$ se fait via action de $\iota(g)$ où $\iota : Z_G^0(\Q_p) \hookrightarrow G(\Q_p)$ est l'injection canonique. De plus $Z_G^0(\Q_p)$ est commutatif, on en déduit que l'action (induite) de $ Z_G^0(\Q_p) / K_Z $ sur $\Sht(G, \mu, b)/K$ et sur $\Sht(Z^0_G, \lambda)/ K_Z$ est bien définie. 
	
D'après la proposition \ref{itm : géométrique}, il y a un isomorphisme $\J(\Q_p)$-équivariant qui commute avec les données de descentes
\[
\Big( \Sht(G, \mu, b) \times_{\underline{Z_G^0(\Q_p)}} \Sht(Z_G^0, \lambda) \Big) / K \longrightarrow \Sht(G, b\overline{b}_{\lambda} ,\mu \cdot \lambda) / K.
\]	 

Il reste à comparer l'espace $\Big( \Sht(G, \mu, b)/K \Big) \times_{\underline{Z_G^0(\Q_p)/ K_Z}} \Big( \Sht(Z_G^0, \lambda)/ K_Z \Big)$ avec le quotient $\Big( \Sht(G, \mu, b) \times_{\underline{Z_G^0(\Q_p)}} \Sht(Z_G^0, \lambda) \Big) / K$. Il suffit de comparer au niveau des points. 

Montrons tout d'abord que le morphisme canonique 
\begin{equation} \phantomsection \label{itm : ensembliste}
	\Big( \Sht(G, \mu, b) \times_{\underline{Z_G^0(\Q_p)}} \Sht(Z_G^0, \lambda) \Big) / K \longrightarrow  \Big( \Sht(G, \mu, b)/K \Big) \times_{\underline{Z_G^0(\Q_p)}} \Sht(Z_G^0, \lambda)
\end{equation}
est un isomorphisme (ensembliste). En effet, $(x, y)$ et $(z,t)$ dans $ \Sht(G, \mu, b) \times \Sht(Z_G^0, \lambda) $ sont dans la même classe dans l'ensemble à gauche si et seulement s'il existe $g \in Z_G^0(\Q_p)$ et $k \in K$ de sorte que $ z = x \cdot g \cdot k $ et $ t = g^{-1} \cdot y $.

De même, $(x, y)$ et $(z,t)$ sont dans la même classe dans l'ensemble à droite si et seulement s'il existe $g \in Z_G^0(\Q_p)$ et $k \in K$ de sorte que $ z = x \cdot k \cdot g $ et $ t = g^{-1} \cdot y $. Puisque $Z_G^0(\Q_p)$ est contenu dans le centre de $G(\Q_p)$, on voit que $g \cdot k = k \cdot g$. On en déduit que le morphisme (\ref{itm : ensembliste}) est un isomorphisme.

L'action de $K_Z$ est triviale sur $\Sht(G, \mu, b)/K$ alors $\Big( \Sht(G, \mu, b)/K \Big) \times_{\underline{Z_G^0(\Q_p)}} \Sht(Z_G^0, \lambda) \simeq \Big( \Sht(G, \mu, b)/K \Big) \times_{\underline{Z_G^0(\Q_p)}} \Sht(Z_G^0, \lambda)/K_Z $. Maintenant, l'action de $K_Z$ étant triviale sur $\Big( \Sht(G, \mu, b)/K \Big) \times \Sht(Z_G^0, \lambda)/K_Z $, on en déduit que
\[
\Big( \Sht(G, \mu, b)/K \Big) \times_{\underline{Z_G^0(\Q_p)}} \Sht(Z_G^0, \lambda)/K_Z \simeq \Big( \Sht(G, \mu, b)/K \Big) \times_{\underline{Z_G^0(\Q_p)/K_Z}} \Sht(Z_G^0, \lambda)/K_Z.
\]

D'autre part, l'action de $Z_G^0(\Q_p) / K_Z$ sur $\Sht(Z_G^0, \lambda)/K_Z$ étant sans point fixe, on en déduit qu'il est de même pour l'action de $Z_G^0(\Q_p) / K_Z$ sur $\big( \Sht(G, \mu, b)/K \big) \times \big( \Sht(Z_G^0, \lambda)/ K_Z \big)$.
\end{proof}
\begin{remarque} \phantomsection \label{itm : remarque}
	Notons $\overline{\Phi}$ le changement de base vers $\Spa(\C_p)^{\diamond}$ du morphisme $\Phi : \Sht(G, \mu, b) \times_{\Spa(\breve{\Q}_p)^{\diamond}} \Sht(Z_G^0, \lambda) \longrightarrow \Sht(G, b\overline{b}_{\lambda} ,\mu \cdot \lambda)$. On remarque que $\overline{\Phi}$ admet une section. En effet, $\C_p$ étant un corps perfectoïde, on peut choisir une modification $\alpha$ dans $\Sht(Z^0_G, \lambda)(\C_p, \C_p^b)$ et on a également $\alpha^{-1} \in \Sht(Z^0_G, \lambda^{-1})(\C_p, \C_p^b) $. Une section de $\overline{\Phi}$ est donnée par
	\[
	\Sht(G,\mu \cdot \lambda, b\overline{b}_{\lambda} ) \ni \beta \longmapsto \big( \beta \times \alpha^{-1}, \alpha \big) \in \Sht(G, \mu, b) \times_{\Spa(\C_p)^{\diamond}} \Sht(Z_G^0, \lambda).
	\] 
	En particulier, on a un isomorphisme (qui, à priori, ne commute pas avec les données de descentes)
	\[
	\Sht(G, \mu, b) \times_{\Spa(\C_p)^{\diamond}} \Sht(Z_G^0, \lambda) \simeq \Sht(G ,\mu \cdot \lambda, b\overline{b}_{\lambda}) \times_{\Spa(\C_p)^{\diamond}} \Sht(Z_G^0, 1).
	\] 
	Il y a également un isomorphisme $G(\Q_p) \times \J_b(\Q_p)$-équivariant entre $\Sht(G, \mu \cdot \lambda, b \overline{b}_{\lambda})_{\C_p}$ et $\Sht(G, \mu, b)_{\C_p}$ donné par 
	\[
	 \Sht(G, \mu, b) \ni \gamma \longmapsto \alpha \times \gamma \in \Sht(G, \mu \cdot \lambda, b \overline{b}_{\lambda})
	\]
	qui, à priori, ne commute pas avec les données de descentes.
\end{remarque}

\section{Groupes de cohomologie de quelques espaces de Rapoport-Zink} \label{itm : application}
Dans ce paragraphe, on va calculer la cohomologie de quelques espaces de Rapoport-Zink non ramifiés en utilisant les résultats des paragraphes précédents et \cite{Boyer09}, \cite{Boyer14}. 

Considérons un triplet $(G, \mu, b)$ où $G$ est un groupe réductif (non ramifié) sur $\Q_p$, $ b \in G(\breve{\Q}_p) $ et $\mu \in X^{+}_{*}(T)$ minuscule avec $b \in B(G, \mu)$. On suppose que $(G, \mu, b)$ correspond à une donnée de Rapoport-Zink de type $EL$ (\cite{RZ96}). À une telle donnée $(G, \mu, b)$, on associe un groupe $p$-divisible $\mathbb{X}$ avec structures additionnelles. Notons $\mathcal{M}(G, \mu, b)$ le foncteur qui classifie les déformations par quasi-isogénies de $\mathbb{X}$ avec structures additionnelles. Ce foncteur est représentable par un schéma formel défini sur $\Spf(\mathcal{O}_{\breve{\Q}_p})$ que l'on notera encore  $\mathcal{M}(G, \mu, b)$. 

Considérons ensuite l'espace rigide $\breve{\mathcal{M}}(G, \mu, b)$ défini sur $\breve{\Q}_p$ associé au schéma formel $\mathcal{M}(G, \mu, b)$. Pour chaque sous groupe compact ouvert $K_p \subset G(\Z_p)$, il existe un espace rigide $\breve{\mathcal{M}}_{K_p}(G, \mu, b)$ qui est défini comme le revêtement étale de $\breve{\mathcal{M}}(G, \mu, b)$ classifiant les $\mathcal{O}_{F}$ trivialisations modulo $K_p$ du module de Tate $p$-adique de groupe $p$-divisible universel sur $\breve{\mathcal{M}}(G, \mu, b)$. On a donc une tour d'espaces rigides $ \big( \breve{\mathcal{M}}_{K_p}(G, \mu, b) \big)_{K_p} $ qui possède à la fois une action de $G(\Q_p) \times J_b(\Q_p) $ et une donnée de descente. 

Il y a un espace de module de Shtukas $\Sht(G, \mu, b)$ associé à chaque donnée $(G, \mu, b)$. On a une identification de diamants sur $\Spa(\breve{\Q}_p)$ : $ \displaystyle \Sht(G, \mu, b) = \mathop{\mathrm{lim}}_{\overleftarrow{K_p}}  \Sht(G, \mu, b) / K_p$.

Puisque $\mu$ est minuscule, d'après le théorème 24.2.5 de \cite{SW17}, on a 
\begin{equation} \phantomsection \label{itm : Shtuka vs RZ}
 \Sht(G, \mu, b) / K_p =  \breve{\mathcal{M}}_{K_p}(G, \mu, b)^{\diamond}.	
\end{equation}

Soit $X$ un $\C_p$-espace analytique tel que $| X |$ est séparé. Un faisceau $\Z_{\ell}$-adique sur $X_{\text{ét}}$ est un système projectif $(\mathcal{F}_n)_{n \in \N}$ de faisceaux en $\Z_{\ell}$-module sur $X_{\text{ét}}$ vérifiant $\ell^n \mathcal{F}_n = 0$. D'après \cite{Far04}, pour un faisceau $\Z_{\ell}$-adique $(\mathcal{F}_n)_{n \in \N}$, on définit le foncteur sections globales à support compact par la formule
\[
\Gamma_c(X, (\mathcal{F}_n)_n) = \{ (s_n)_n \in \mathop{\mathrm{lim}}_{\leftarrow} \Gamma (X, \mathcal{F}_n) | \overline{\cup_n \text{supp}(s_n)} \text{ est compact} \}.
\]

On posera $ \displaystyle H_c^q(X, (\mathcal{F}_n)_n) = R^q\Gamma_c(X, \bullet )((\mathcal{F}_n)_n) $. \\

	Soit $K_p \subset G(\Z_p) $ un niveau. D'après \cite{Far04}, on peut montrer que
	\[
	H_c^{\bullet}(\breve{\mathcal{M}}_{K_p}(G, \mu, b), \Z_{\ell}) = \mathop{\mathrm{lim}}_{\overrightarrow{V}} \mathop{\mathrm{lim}}_{\overleftarrow{n}} H_c^{\bullet}(V \otimes_{\breve{\Q}_p} \C_p, \Z / \ell^n \Z )
	\]
	où $V$ parcourt les ouverts relativement compacts de $\breve{\mathcal{M}}_{K_p}(G, \mu, b)$ et où $\ell \neq p$ est un nombre premier. On note également
	\[
	H_c^{\bullet}(\breve{\mathcal{M}}_{K_p}(G, \mu, b), \Q_{\ell}) := H_c^{\bullet}(\breve{\mathcal{M}}_{K_p}(G, \mu, b), \Z_{\ell}) \otimes \overline{\Q}_{\ell}.
	\]

\begin{exemple} (\cite{Far04})
	Soit $Z$ un tore non ramifié et $K_Z \subset Z(\Q_p)$ un sous groupe compact. Alors, on a
	\[
	H^{q}_c \big( \breve{\mathcal{M}}_{K_Z}(Z, 1, \Id), \Q_{\ell} \big) = \left\lbrace \begin{array}{cc}
	\mathcal{C}_c^{\infty} \big( Z(\Q_p) / K_Z, \overline{\Q}_{\ell} \big) & \text{si q = 0} \\
	0 & \text{sinon}.
	\end{array} \right.
	\]
	et donc 
	\[
	\mathop{\mathrm{lim}}_{\overrightarrow{K_Z}} H^{0}_c \big( \breve{\mathcal{M}}_{K_Z}(Z, 1, \Id), \Q_{\ell} \big) = \mathcal{C}_c^{\infty} (Z(\Q_p), \overline{\Q}_{\ell}) 
	\]
	où action de $\J_b(\Q_p) = \Z(\Q_p)$ se fait par la représentation régulière. L'action du groupe de Weil $W_E$ est l'action triviale. D'après la remarque $\ref{itm : remarque}$, il y a un isomorphisme $Z(\Q_p) \times \Z(\Q_p)$-équivariant $\Sht(Z, \lambda, \lambda(p))_{\C_p} \xrightarrow{\sim} \Sht(Z, 1, \Id)_{\C_p}$. Il y a donc des isomorphismes $Z(\Q_p)$-équivariants 
\[
H^{q}_c \big( \breve{\mathcal{M}}_{K_Z}(Z, 1, \Id), \Q_{\ell} \big) = H^{q}_c \big( \breve{\mathcal{M}}_{K_Z}(Z, \lambda, \lambda(p)), \Q_{\ell} \big).
\]	
\end{exemple}

On considère maintenant les triplets $(G, \mu, b)$ tels que
\begin{enumerate}
	\item[$\bullet$] $G = \Res_{F/\Q_p} GL_{F}(V)$ où $F$ est l'unique extension non ramifiée de degré $d$ de $\Q_p$ et $V$ est un $F$-espace vectoriel de dimension $n$,
	\item[$\bullet$] $\mu$ est un cocaractère minuscule de $G$. Un tel $\mu$ est déterminé par des couples d'entiers $(p_{\tau}, q_{\tau})_{\tau \in I}$ où $I = \hom_{\Q_p}(F, \overline{\Q}_p)$. On suppose de plus que 
	\begin{enumerate}
		\item[$\bullet$] Il existe $\tau_0 \in I$ tel que $(p_{\tau_0}, q_{\tau_0}) = (1, n-1)$,
		\item[$\bullet$] Pour $ \tau \neq \tau_0 $ on a $(p_{\tau}, q_{\tau}) = (0, n)$ ou $(p_{\tau}, q_{\tau}) = (n, 0)$.
	\end{enumerate}
	\item[$\bullet$] $b \in B(G, \mu)$ est l'unique élément basique.
\end{enumerate}

Pour un tel triplet, on a $\J_b(\Q_p) = D^{\times}_{ n / F}$, le groupe des inversibles de l'algèbre de division d'invariant $\frac{1}{n}$, de centre $F$. Le corps de définition de $\mu$ est $F$. On note encore $J \subset I \setminus \{ \tau_0 \} $ l'ensemble de $\tau \neq \tau_0$ tel que $ (p_{\tau}, q_{\tau}) = (n,0) $. Les espaces de Rapoport-Zink associés sont non ramifiés de type $EL$. Par la suite, on notera $\Sht(\mu)$ (resp. $\breve{\mathcal{M}}^{\mu}_K$) pour $\Sht(G, \mu, b)$ (resp. $\breve{\mathcal{M}}_K(G, \mu, b)$). Lorsque la signature $\mu_{\mathcal{LT}} = (1, n-1), (0, n), \cdots, (0, n)$, on retrouve la tour de Lubin-Tate. Pour $\sigma$ une représentation irréductible de $D^{\times}_{ n / F}$, on notera $H_c^q(\breve{\mathcal{M}}_K^{\mu})[\sigma] = \hom_{D^{\times}_{ n / F}} \big(  H_c^q( \breve{\mathcal{M}}^{\mu}_K, \Q_{\ell}), \sigma \big) $.
\begin{definition} \cite{Boyer09}
	\begin{enumerate}
		\item[-] Soient $\pi_1$ et $\pi_2$ des représentations de respectivement $GL_{n_1}(F)$ et $GL_{n_2}(F)$, on note $\pi_1 \times \pi_2$ l'induite parabolique $ \Ind_{P_{n_1, n_1 + n_2}(F)}^{GL_{n_1 + n_2}(F)} (\pi_1\{n_2 / 2\} \otimes \pi_2\{-n_1/2\}) $.
		\item[-] Soit $g$ un diviseur de $n = sg$ et $\pi$ une représentation cuspidale irréductible de $GL_g(F)$ : 
		\begin{enumerate}
			\item[$\bullet$] $\pi\{\frac{1-s}{2}\} \times \pi\{\frac{3-s}{2}\} \times \cdots \times \pi\{\frac{s-1}{2}\}$ possède un unique quotient (resp. sous espace) irréductible. C'est une représentation de Steinberg (resp. de Speh) généralisée notée habituellement $\St_s(\pi)$ (resp. $\Speh_s(\pi)$). 
			\item[$\bullet$] $\St_{s-i}(\pi)\{\frac{-i}{2}\} \times \Speh_i(\pi)\{\frac{s-i}{2}\}$ possède un unique sous espace irréductible que l'on note $ \LT_{\pi}(s,i) $. En particulier, pour $i = 0$ (resp. $i = s-1$), on retrouve $\St_s(\pi)$ (resp. $\Speh_s(\pi)$)
		\end{enumerate}
	\item[-] Pour $\pi$ une représentation irréductible cuspidale de $GL_g(F)$ et $t > 0$, $\pi[t]_D$ désignera la représentation $\JL^{-1}(\St_t(\pi))^{\vee}$ de $D^{\times}_{n / F}$. 
	\end{enumerate}	
\end{definition} 
\begin{théorème}
		 Pour tout diviseur $g$ de $n = gs$ et toute représentation irréductible cuspidale $\pi$ de $GL_g(F)$, on a alors des isomorphismes $G(\Q_p) \times W_{F}$-équivariants
		\[
		\mathop{\mathrm{lim}}_{\overrightarrow{K}} H_c^{n-1-i}(\breve{\mathcal{M}}_K^{\mu})[\pi[s]_D] = \left\lbrace \begin{array}{cc}
		\displaystyle \LT_{\pi}(s,i) \otimes \mathcal{L}(\pi) | \cdot |^{ - \frac{s(g+1) - 2(i+1)}{2}} \cdot \prod_{\tau \in J}  \omega_i \circ (\prescript{\tau}{}{\rec}^{-1}_{F}) & 0 \leq i < s \\
		0 & i < 0
		\end{array} \right. 
		\]
		où $\omega_i$ est le caractère central de $\LT_{\pi}(s,i)$ et $\rec^{-1}_{F}$ est le morphisme de réciprocité d'Artin et où $\prescript{\tau}{}{\rec}^{-1}_{F} := \tau \cdot \rec^{-1}_{F} \cdot \tau^{-1}$.
\end{théorème}
\begin{remarque}
	Tout énoncé valable sur l'espace de Lubin-Tate trouvera son analogue dans la situation plus générale précédente. En particulier, dans \cite{Boyer14} il est annoncé que les $ \displaystyle \mathop{\mathrm{lim}}_{\overrightarrow{K}} H_c^{\bullet}(\breve{\mathcal{M}}^{\mu_{\mathcal{LT}}}_{K_p}, \Z_{\ell}) $ sont $\Z_{\ell}$-libres, ce qui impliquerait
	la même propriété pour $ \displaystyle \mathop{\mathrm{lim}}_{\overrightarrow{K}} H_c^{\bullet}(\breve{\mathcal{M}}^{\mu}_{K_p}, \Z_{\ell}) $.
\end{remarque}
\begin{proof}
On exploitera les résultats obtenus dans les sections précédentes pour calculer la cohomologie de la tour $ \big( \breve{\mathcal{M}}^{\mu}_{K_p} \big)_{K_p} $. Il s'agit de trouver les relations de ce dernier avec la tour de Lubin-Tate. Tout d'abord, d'après \cite{Boyer09}, on a des isomorphismes $G(\Q_p) \times W_{F}$-équivariants
\[
\mathop{\mathrm{lim}}_{\overrightarrow{K}} H_c^{n-i-1}(\breve{\mathcal{M}}_K^{\mu_{\mathcal{LT}}})[\pi[s]_D] = \left\lbrace \begin{array}{cc}
\LT_{\pi}(s,i) \otimes \mathcal{L}(\pi) | \cdot |^{ - \frac{s(g+1) - 2(i+1)}{2}} & 0 \leq i < s \\
0 & i < 0.
\end{array} \right.
\]

Le groupe $G = \Res_{F/\Q_p} GL_{F}(V)$ est déployé sur $F$ et $\displaystyle G_{F}(F) \simeq \prod_{\tau \in I} GL_n(F)$. Le centre de $G$ est $ Z = \Res_{F/\Q_p} GL_1$ et il est aussi déployé sur $F$ : $ \displaystyle Z_{F} \simeq \prod_{\tau \in I} GL_1(F)$.
Identifions $I$ avec $\Z / d\Z$, le groupe de caractère est donné par
\[
X_{*}(Z) = \Big\{ (x_{ij})_{i \in \Z / d\Z, 1\leq j \leq n}  \ \vert \ \forall i,j \ x_{i,j} \in \Z \Big\}.
\]
Le groupe de Weil associé est $ W = S^d_n$. Le cocaractère $\mu$ s'écrit sous la forme $ (x_{i,j})_{i \in \Z / d\Z, 1\leq j \leq n} $
\[
x_{i,j} = \left\lbrace \begin{array}{ccc}
(1, 0, \cdots, 0) & \mathop{\mathrm{si}}\nolimits \ i = i_0 \\
(1, 1, \cdots, 1) & \mathop{\mathrm{si}}\nolimits \ i \in J \\
(0, 0, \cdots, 0) & \mathop{\mathrm{sinon.}}\nolimits
\end{array} \right. 
\]

Considérons maintenant le cocaractère $\lambda = (z_{i,j})_{i \in \Z / d\Z, 1\leq j \leq n}$ où $z_{i,j} = \left\lbrace \begin{array}{cc}
1 & \mathop{\mathrm{si}}\nolimits \ i \in J \\
0 & \mathop{\mathrm{sinon.}}\nolimits
\end{array} \right. $

Il est clair que $\lambda(p)$ est dans le centre de $G(\breve{\Q}_p)$ et d'autre part, on constate que $\mu = \mu_{\mathcal{LT}} \cdot \lambda$. On voit également que $ b_{\mathcal{LT}} \cdot \lambda(p) $ est l'unique classe basique de $B(G, \mu)$. D'après la remarque \ref{itm : remarque}, il y a un isomorphisme $G(\Q_p) \times \J_b(\Q_p)$-équivariant
\[
\Sht(\mu)_{\C_p} \xrightarrow{\sim} \Sht(\mu_{\mathcal{LT}})_{\C_p}. 
\]	
Pour toute représentation supercuspidale $\pi$ de $GL_g(F)$, il y a alors des isomorphismes $G(\Q_p)$-équivariants
\begin{equation} \phantomsection \label{itm : sans action galoisienne}
	\mathop{\mathrm{lim}}_{\overrightarrow{K}} H_c^q(\breve{\mathcal{M}}_K^{\mu})[\pi[s]_D] = \mathop{\mathrm{lim}}_{\overrightarrow{K}} H_c^q(\breve{\mathcal{M}}_K^{\mu_{\mathcal{LT}}})[\pi[s]_D].
\end{equation}
En particulier, on a $ H_c^{n-i-1}(\breve{\mathcal{M}}^{\mu}_K)[\pi[s]_D] = 0 $ pour $i < 0$.

D'après le corollaire \ref{itm : niveau fini}, il y a un isomorphisme $D^{\times}_{ n / F}$-équivariant d'espaces rigides qui commute avec la donnée de descente
\begin{equation*}
\breve{\mathcal{M}}_K^{\mu_{\mathcal{LT}}} \times_{\underline{Z(\Q_p)/ K_Z}}   \breve{\mathcal{M}}_{K_Z}(Z, \lambda, \lambda(p))  \longrightarrow  \breve{\mathcal{M}}^{\mu}_K.	
\end{equation*}
On en déduit qu'il y a un isomorphisme $D^{\times}_{ n / F} \times W_F$-équivariant
\begin{equation} \phantomsection \label{itm : produit tensoriel dérivé}
R\Gamma_c \Big( (\breve{\mathcal{M}}_K^{\mu_{\mathcal{LT}}})_{\C_p} \times \breve{\mathcal{M}}_{K_Z}(Z, \lambda, \lambda(p))_{\C_p}, \Q_{\ell} \Big) \otimes^{\mathbb{L}}_{\Q_{\ell}[Z(\Q_p)/ K_Z]} \Q_{\ell} \xrightarrow{\sim} R\Gamma_c \Big( (\breve{\mathcal{M}}^{\mu}_K)_{\C_p}, \Q_{\ell} \Big).	
\end{equation}
D'après la remarque \ref{itm : remarque}, on voit que $(\breve{\mathcal{M}}^{\mu_{\mathcal{LT}}}_K)_{\C_p} \times   \breve{\mathcal{M}}_{K_Z}(Z, \lambda, \lambda(p))_{\C_p}$ est un $Z(\Q_p)/ K_Z$-torseur trivial au-dessus de $(\breve{\mathcal{M}}^{\mu}_K )_{\C_p}$. Alors, le produit tensoriel dérivé dans (\ref{itm : produit tensoriel dérivé}) dégénère en un produit tensoriel. Puisque les groupes de cohomologie supérieure de $\breve{\mathcal{M}}_{K_Z}(Z, \lambda, \lambda(p))_{\C_p}$ s'annulent, la formule de Kunneth implique qu'il y a un isomorphisme $ D^{\times}_{ n / F} \times W_F $-équivariant
\[
H^q_c \Big( (\breve{\mathcal{M}}^{\mu_{\mathcal{LT}}}_K)_{\C_p} \times \breve{\mathcal{M}}_{K_Z}(Z, \lambda, \lambda(p))_{\C_p}, \Q_{\ell} \Big) \simeq H^q_c \Big( (\breve{\mathcal{M}}^{\mu_{\mathcal{LT}}}_K)_{\C_p}, \Q_{\ell} \Big) \times H^0_c \Big( \breve{\mathcal{M}}_{K_Z}(Z, \lambda, \lambda(p))_{\C_p}, \Q_{\ell} \Big)
\]
On en déduit qu'il y a un isomorphisme $W_F$-équivariant entre $  H_c^{q}( \breve{\mathcal{M}}^{\mu}_K)[\pi[s]_D]$ et
\[
\Big( H_c^{q}( \breve{\mathcal{M}}^{\mu_{\mathcal{LT}}}_K)[\pi[s]_D] \times \hom( H^0_c ( \breve{\mathcal{M}}_{K_Z}(Z, \lambda, \lambda(p)), \Q_{\ell}) , \Q_{\ell} ) \Big) \otimes_{\Q_{\ell}[Z(\Q_p)/ K_Z]} \Q_{\ell}.
\]
D'après \cite{Chen}, il y a des isomorphismes $Z(\Q_p) \times W_{F}$-équivariants
\[
\hom_{Z_{\lambda}(\Q_p)} \big(  H^0( \breve{\mathcal{M}}_{K_Z}(Z, \lambda, \lambda(p)), \Q_{\ell}), \omega_i \big) = \omega_i \otimes \prod_{\tau \in J} \omega \circ ( \prescript{\tau}{}{\rec}^{-1}_{F} )
\] 
On en déduit que pour $0 \leq i < s$, il y a une inclusion
\begin{equation} \phantomsection \label{itm : inclusion niveau fini}
	\LT_{\pi}(s,i)^K \otimes \mathcal{L}(\pi) | \cdot |^{-\frac{s(g+1) - 2(i+1)}{2}} \cdot \prod_{\tau \in J}  \omega_{\pi} \circ ( \prescript{\tau}{}{\rec}^{-1}_{F} ) \hookrightarrow  H_c^{n-i-1}(\breve{\mathcal{M}}^{\mu}_K)[\pi[s]_D].
\end{equation}
Mais l'isomorphisme (\ref{itm : sans action galoisienne}) implique que l'inclusion (\ref{itm : inclusion niveau fini}) est en fait un isomorphisme de représentations de $W_F$. En prenant la limite projective lorsque le niveau $K$ varie, on en déduit que, pour $ 0 \leq i < s $,
\[
\mathop{\mathrm{lim}}_{\overrightarrow{K}} H_c^{n-i-1}(\breve{\mathcal{M}}^{\mu}_K)[\pi[s]_D] = \LT_{\pi}(s,i) \otimes \mathcal{L}(\pi) | \cdot |^{-\frac{s(g+1) - 2(i+1)}{2}} \cdot \prod_{\tau \in J}  \omega_{\pi} \circ ( \prescript{\tau}{}{\rec}^{-1}_{F}). 
\]
\end{proof}

\thispagestyle{fancy}
\fancyhf{}
\lfoot{NGUYEN Kieu Hieu \quad $\bullet$ \quad \textit{Email} kieuhieu.nguyen@math.univ-paris13.fr, Université Paris 13, Sorbonne Paris-Cité, LAGA, CNRS, UMR 7539, F-93430, Villetaneuse, FRANCE, PerCoLarTor, ANR-14-CE25}

\end{document}